\documentclass[10pt]{amsart}
\usepackage{amssymb,amsfonts}
\usepackage{amsmath,amscd}
\usepackage[all]{xy}
\usepackage{enumitem}
\usepackage{mathrsfs}
\usepackage{ stmaryrd }
\usepackage{url}
\usepackage
[pdfauthor={},
 pdftitle={},
 bookmarks=false]
{hyperref}
\usepackage{bbm}

\newcounter{qcounter}

\newcommand\define{\newcommand}
\define\ur{\mathrm{ur}}
\define\reg{\mathrm{reg}}

\define\bP{\mathbb{P}}

\define\isoto{\xrightarrow{\sim}}
\define\onto{\twoheadrightarrow}

\newcommand{\dia}[1]{{\langle #1 \rangle}}

\newcommand{\ttmat}[4]{\left( \begin{array}{cc}
#1 & #2 \\
#3 & #4
\end{array}
\right)}

\newcommand{\Z}{\mathbb{Z}}
\newcommand{\Q}{\mathbb{Q}}

\newcommand{\R}{\mathbb{R}}
\newcommand{\F}{\mathbb{F}}

\newcommand{\sO}{\mathcal{O}}

\newcommand{\m}{\mathfrak{m}}

\newcommand{\Hom}{\mathrm{Hom}}
\newcommand{\Gal}{\mathrm{Gal}}

\newcommand{\Ext}{\mathrm{Ext}}

\newcommand{\End}{\mathrm{End}}

\newcommand{\Fr}{\mathrm{Fr}}

\newcommand{\lb}{{[\![}}
\newcommand{\rb}{{]\!]}}

\newcommand{\red}{\mathrm{red}}

\define\GL{{\mathrm{GL}}}

\define\SL{{\mathrm{SL}}}
\define\kcyc{\kappa_{\mathrm{cyc}}}

\define{\Fitt}{\mathrm{Fitt}}
\define{\Ann}{\mathrm{Ann}}

\newtheorem{thm}{Theorem}[subsection] 
\newtheorem*{thm*}{Theorem}

\newtheorem*{claim}{Claim}
\newtheorem{cor}[thm]{Corollary}
\newtheorem{prop}[thm]{Proposition}
\newtheorem{lem}[thm]{Lemma}
\newtheorem{conj}[thm]{Conjecture}

\theoremstyle{definition}
\newtheorem{defn}[thm]{Definition}

\newtheorem{eg}[thm]{Example}

\theoremstyle{remark}
\newtheorem{rem}[thm]{Remark}
\newtheorem{rems}[thm]{Remarks}

\newcommand{\bT}{\mathbb{T}}

\newcommand{\Db}{{\bar D}}

\usepackage{color}

\newcommand{\tr}{{\mathrm{tr}}}

\newcommand{\RG}{\mathrm{R}\Gamma}
\newcommand{\sv}[2]{\ensuremath{\big(\begin{smallmatrix}#1 \\ #2 \end{smallmatrix}\big)}}

\makeatletter
\let\c@equation\c@thm
\makeatother
\numberwithin{equation}{subsection}

\title[Weight $k$ Eisenstein ideal and tame Bloch-Kato]{The Eisenstein ideal for weight $k$ and a Bloch-Kato conjecture for tame families}

\author{Preston Wake}
\address{Department of Mathematics, Michigan State University \\
East Lansing, MI 48824}
\email{wakepres@msu.edu}

\begin{document}

\begin{abstract}
We study the Eisenstein ideal for modular forms of even weight $k>2$ and prime level $N$. We pay special attention to the phenomenon of \emph{extra reducibility}: the Eisenstein ideal is strictly larger than the ideal cutting out reducible Galois representations. We prove a modularity theorem for these extra reducible representations. As consequences, we relate the derivative of a Mazur-Tate $L$-function to the rank of the Hecke algebra, generalizing a theorem of Merel, and give a new proof of a special case of an equivariant main conjecture of Kato. In the second half of the paper, we recall Kato's formulation of this main conjecture in the case of a family of motives given by twists by characters of conductor $N$ and $p$-power order and its relation to other formulations of the equivariant main conjecture.
\end{abstract}
\maketitle
\tableofcontents

\section{Introduction}

\subsection{Summary} Mazur initiated the study of the Eisenstein ideal in the context of modular forms of weight 2 and prime level $N$ as a powerful tool for studying the arithmetic of modular curves and $L$-functions \cite{mazur1978}. In this context, the Eisenstein ideal measures congruences modulo $p$ between the Eisenstein series and a cusp form that occur because $p$ divides an Euler factor in the $L$-function that is the constant term of the Eisenstein series.

This paper grew out of an attempt to unify two approaches for answering a question of Mazur on the $\Z_p$-rank of the Eisenstein ideal \cite[Section II.19, pg.~140]{mazur1978}. The first approach, starting with Merel \cite{merel1996} and more recently Lecouturier \cite{lecouturier2017arxiv}, is analytic, and relates the rank to the order of vanishing of an $L$-function. The second approach, starting with Calegari--Emerton \cite{CE2005} and more recently the author with Wang-Erickson \cite{PG3}, relates the rank to class groups or Galois cohomology of characters.

Although the analytic and algebraic approaches seem completely different, we identify a theme that is central to both approaches: the idea of \emph{extra reducibility}. In \cite{lecouturier2017arxiv}, this idea manifests itself in the existence of extra mod-$p$ Eisenstein series of level $\Gamma_0(N)$ when $p$ divides $N-1$. In \cite{PG3}, it manifests itself in the existence of first-order deformations of the residual representation that are still reducible. 

In the first part of this paper, we explore the theme of extra reducibility in the context of modular forms of even weight $k>2$. We compute the Galois deformation ring parameterizing the reducible deformations. We show that these reducible deformations are all accounted for by extra Eisenstein series in characteristic $p$. We think of this as a `reducible modularity' theorem. As a consequence, we prove that the obstruction to deforming the mod-$p$ Eisenstein series \emph{as a cusp form} is given by an equivariant $L$-function that we call the Mazur-Tate $\zeta$-function $\xi_\mathrm{MT}$. We use this to prove our main result, which relates the rank of the Eisenstein ideal to the order of vanishing of $\xi_\mathrm{MT}$, generalizing a theorem of Merel \cite{merel1996} to higher weight. In the case where this order of vanishing is one, we relate the value of the leading term, an analytic invariant, to an algebraic invariant in Galois cohomology.

In the second part of the paper, we leave behind modular forms and discuss the conjecture framework concerning relations between the analytic and algebraic invariants of the type mentioned in the last sentence of the previous paragraph. As will be unsurprising to experts, these relations are ultimately predicted by an equivariant version of the Iwasawa main conjecture. However, this relation is not totally transparent. We derive the relation from first principles using Kato's formulation of the main conjecture \cite{kato1993,kato1993a}
\footnote{Kato's main conjecture is a reformulation of the Bloch-Kato conjecture \cite{BK1990} that is suitable for considering families of motives. A similar reformulation was found independently, around the same time, by Fontaine and Perrin-Riou \cite{FP1994}. We focus on Kato's formulation because of the attention he pays to integral aspects of the theory.}
, specialized to the case of `tame families'. 
Using a method of ``changing Selmer conditions", we show that our results are equivalent to a 
version of the equivariant main conjecture formulated by Greither and Popescu \cite{GP2015} (which has already been proven).  We end with a discussion of several equivalent forms of the conjecture in terms of: Fitting ideals of cohomology, obstructions to lifting residual cohomology classes, cup products, and slopes of cohomology classes (or ``$\mathcal{L}$-invariants").

Our results in the second part concern proving that various formulations of the main conjecture are equivalent. We emphasize that our methods of the first part only give a new proof of the main conjecture; we do not prove any new cases. Other proofs have been given by Coates--Sinnot \cite{CS1974}, using Stickelberger's theorem, and by Greither--Popescu \cite{GP2015}, using the main conjecture for totally real fields proven by Wiles \cite{wiles1990}. However, unlike known proofs, we do not use $p$-adic methods, which significantly simplifies the proofs.

Kato's insights about the importance of $p$-adic Hodge theory in the study of special values of $L$-functions have led to an emphasis on $p$-adic aspects of the theory in most expositions. We hope that our explication of Kato's ideas in the tame case, where $p$-adic Hodge theory plays no special role, can be of expository value.
We believe that this method, using tame families, is quite versatile. It may be possible to apply these techniques to study main conjectures for other motives, or, in cases where the main conjecture is known, to prove finer results.

\subsubsection{Extra reducibility for $X_0(11)$} Before we discuss our results in more detail, we illustrate the idea of extra reducibility in the simplest case: the modular curve $X_0(11)$ (which is an elliptic curve). As made famous by Mazur \cite{mazur1978}, the Frobenius traces on $X_0(11)$ satisfy
\begin{equation}
\label{eq:X0(11) reducible}
a_\ell(X_0(11)) \equiv 1 + \ell \pmod{5}
\end{equation}
for all primes $\ell \ne 11$. There are two (related) explanations for this congruence:
\begin{description}
\item[Galois] the Galois representation $X_0(11)[5]$ is reducible,
\item[Modular] the cusp form $f_{X_0(11)}$ associated to $X_0(11)$ is congruent modulo $5$ to the Eisenstein series of weight $2$ and level $11$.
\end{description}
The congruence \eqref{eq:X0(11) reducible} can be called \emph{reducibility} for $X_0(11)$. However, there is a stronger congruence
\begin{equation}
\label{eq:X0(11)}
a_\ell(X_0(11)) \equiv \chi(\ell) + \chi^{-1}(\ell)\ell \pmod{25}
\end{equation}
where $\chi:(\Z/11\Z)^\times \to (\Z/25\Z)^\times$ is the unique character taking the primitive root $2$ to $6$. This congruence also has a Galois-theoretic explanation:
\begin{description}
\item[Galois II] the Galois representation $X_0(11)[5]$ is reducible \emph{and semi-simple}.
\end{description}
Using the theory of lattices (as in Ribet's Lemma \cite[Proposition 2.1]{ribet1976}), this semi-simplicity implies that $a_\ell(X_0(11))$ must satisfy a congruence like \eqref{eq:X0(11)} for some character $\chi$; finding which character is then a simple computation. However, there is no obvious modular explanation for  \eqref{eq:X0(11)}: the right-hand side of the congruence is not the reduction of the $\ell$th Fourier coefficient of an Eisenstein series.  We call this kind of congruence \emph{extra reducibility}, for the kind reducibility not caused by congruence with an Eisenstein series.

Even though \textbf{Galois II} can be used to prove the congruence \eqref{eq:X0(11)}, this proof is unsatisfying to us for two reasons. The first is that \eqref{eq:X0(11)} is \emph{lattice-independent}, in that it is true not just for $X_0(11)$ but for any elliptic curve that is rationally-isogenous to it. But \textbf{Galois II} is lattice-dependent: it is true only for $X_0(11)$. We would prefer to have a lattice-independent proof of a lattice-independent fact. The second reason is that \textbf{Galois II} only explains that \eqref{eq:X0(11)} is true for \emph{some} character $\chi$, and gives no insight into why it is true for the \emph{particular} character $\chi$. A number theorist may like to theorize about the number $6$:  why does $\chi$ send $2$ to $6$ and not $11$ or $16$?

In this paper, our goal is to:
\begin{itemize}
\item generalize the formula \eqref{eq:X0(11)} to modular forms of higher weight (see \eqref{eq:mod p2} below),
\item give a modular and lattice-independent explanation for this formula, and
\item explain the arithmetic significance of the character $\chi$ that appears.
\end{itemize}
The character $\chi$ is significant both \emph{algebraically}, in that it controls a delicate invariant in Galois cohomology, and \emph{analytically}, in that the values is related to special values of $L$-functions. Another goal is to explain that the relation between the these algebraic and analytic invariants is predicted by a special case of a conjecture of Kato (see \cite[Iwasawa Main Conjecture (4.9)]{kato1993a} and Section \ref{subsec:kato conj}) and that extra reducibility can be used to prove this special case.

\begin{rems}[On the history of $X_0(11)$]
The history of the above results is difficult for us to sort out because, although not much was published about this before Mazur's landmark paper \cite{mazur1978}, it is certain that this particular case was understood earlier. Shimura studied $X_0(11)$ extensively, and had access to computations of $a_\ell(X_0(11))$ by Trotter \cite{shimura1966}. It's unclear whether he knew \textbf{Galois II} or the congruence \eqref{eq:X0(11)}, but he knew how to construct two complementary sub-representations of $X_0(11)[5]$, using the cusps and using the cover $X_1(11) \to X_0(11)$, respectively. This latter construction was written about in \cite[Remark 7.27, pg.~196]{shimura1971}, but the first reference that discusses the non-trivial Galois action seems to be by Ogg in 1973  \cite[pg.~ 230]{ogg1973}. Mazur \cite[Proposition II.18.9, pg.~138]{mazur1978} gave a generalization of \textbf{Galois II} to $X_0(N)$ for primes $N$, and named the \emph{Shimura subgroup} after Shimura's work. Mazur attributes the first proof of \eqref{eq:X0(11)} to Serre \cite[pg.~139]{mazur1978}.
\end{rems}

\subsection{Eisenstein ideal for weight $k$ forms} 
\label{subsec:intro eis}
For the entire paper, we fix a triple of integers $(k,p,N)$ such that
\begin{itemize}
\item $k>2$ is an even integer,
\item $p$ is a prime such that $\zeta(1-k) \in \Z_{(p)}^\times$,
\item $N$ is a prime with $p\mid (N-1)$.
\end{itemize}
Note that $\zeta(1-k) \in \Z_{(p)}$ if and only if $(p-1) \nmid k$ by the von Staudt--Clausen Theorem (see \cite[Theorem 5.10, pg.~56]{washington1997}). If $\zeta(1-k) \in \Z_{(p)}$, then $\zeta(1-k) \in \Z_{(p)}^\times$ if $p$ is a regular prime. 

To simplify this introduction, we assume in addition that $p^2 \nmid (N-1)$ and that $p \nmid k(k-1)$. For example, we may take $(k,p,N)=(14,5,11)$. The general results are stated and proven in Section \ref{sec:R=T} below.

\subsubsection{Reducible modularity} Let $\bT$ denote the completion of the Hecke algebra acting on weight-$k$ forms of level $\Gamma_0(N)$ at the $p$-Eisenstein ideal. Let $\bT^0$ denote the quotient acting on cusp forms, and let $I^0 \subset \bT^0$ be its Eisenstein ideal.

As in \cite{PG3}, we study $\bT$ by considering a Galois deformation ring $R_\Db$ for the residual pseudorepresentation $\Db=\F_p(k-1) \oplus \F_p$ of the Eisenstein series. We construct a surjective $\Z_p$-algebra homomorphism $R_\Db \onto \bT$ that we expect is an isomorphism. We do not attempt to prove this (although see Remark \ref{rem:R=T}), but we focus instead on proving a weaker \emph{reducible} modularity theorem. Let $R_\Db^\red$ denote the quotient of $R_\Db$ parameterizing deformations that are reducible, and let $\bT^\red = \bT \otimes_{R_\Db} R_\Db^\red$. Let $\Lambda=\Z_p[(\Z/N\Z)^\times \otimes_\Z \Z_p]$ and $\Lambda_1=\Lambda/I_\mathrm{Aug}^2$, where $I_\mathrm{Aug}$ is the augmentation ideal.

\begin{thm}
\label{thm:intro Rred=Tred}
The map $R_\Db \to \bT$ induces an isomorphism $R_\Db^\red \isoto \bT^\red$, and both $R_\Db^\red$ and $\bT^\red$ are isomorphic to $\Lambda_1$ as $\Z_p$-modules.
\end{thm}

The isomorphism $\bT^\red \to \Lambda_1$ in the theorem comes from a modular eigenform $\tilde{E}_{k,N}$ with coefficients in $\Lambda_1$ that we call the \emph{deformation Eisenstein series}, whose base-change to $\Z_p$ is a usual Eisenstein series $E_{k,N}$ of level $\Gamma_0(N)$. The construction of $\tilde{E}_{k,N}$ is inspired by \cite{lecouturier2017arxiv}. The idea is to take linear combinations of Eisenstein series $E_k(\chi,1)$ and $E_k(1,\chi)$ of level $\Gamma_1(N)$ and show that, if $\chi$ is an infinitesimal deformation of the trivial character, then these linear combinations can descend to $\Gamma_0(N)$.

\begin{rem}[On the Shimura subgroup]
Unlike in previous analytic studies of the Eisenstein ideal \cite{mazur1978,merel1996,lecouturier2017arxiv}, no special role is played in this paper by the Shimura cover $X_1(N) \to X_0(N)$. Indeed, the significance of this cover seems special to weight $2$ and we do not know a weight-$k$ analog. Given this, the use of $\Gamma_1(N)$-structure in the construction of $\tilde{E}_{k,N}$ seems ad hoc. The important phenomenon is that a representation that a priori has deeper level-$N$ structure in fact has $\Gamma_0(N)$-invariants, and this phenomenon appears to be quite general. We plan to study generalizations in the future.
\end{rem}

\subsubsection{The cuspidal-reducible locus and the Mazur-Tate $\zeta$-function} Considering $\Lambda$ as the ring of functions on a Dirichlet character of modulus $N$ and $p$-power order, we have the function
\[
\chi \mapsto L(1-k,\chi)
\]
which we call the \emph{Mazur-Tate $\zeta$-function} after \cite{MT1987}, and denote by $\xi_\mathrm{MT}\in \Lambda$. The image $\xi_\mathrm{MT}(\mathbbm{1})\in \Z_p$ of $\xi_\mathrm{MT}$ under the augmentation map is $\zeta(1-k)(1-N^{k-1})$. The \emph{derivative} $\xi_\mathrm{MT}' \in \F_p$ is the image of $\xi_\mathrm{MT}-\xi_\mathrm{MT}(\mathbbm{1})$ under the isomorphism
\[
I_\mathrm{Aug}/I_\mathrm{Aug}^2 \isoto (\Z/N\Z)^\times \otimes \Z_p \xrightarrow{\log_N} \F_p.
\]
where $\log_N$ is a fixed choice of isomorphism (the `discrete logarithm'). Explicitly,
\[
\xi_\mathrm{MT}' = \frac{1}{k} \sum_{i=1}^{N-1}  B_k(i)\log_N(i)
\]
where $B_k(x)$ is the Bernoulli polynomial.

The constant term $\xi_\mathrm{MT}^\mathrm{Eis}:=a_0(\tilde{E}_{k,N}) \in \Lambda_1$ is closely related to $\frac{1}{2}\xi_\mathrm{MT}$. It has the same derivative, but its constant term is $\frac{1}{2}\zeta(1-k)(1-N^{k/2})$. Let $\bT^{0,\red} =\bT^0 \otimes_{\bT} \bT^\red$. We prove the following.

\begin{thm}
\label{thm:intro T0red}
There is an isomorphism $\bT^{0,\red} \cong \Lambda_1/\xi_\mathrm{MT}^\mathrm{Eis}$.
\end{thm}

When $\xi_\mathrm{MT}' \ne 0$, there is a explicit isomorphism $\Lambda_1/\xi_\mathrm{MT}^\mathrm{Eis} \cong \Z/p^2\Z$. In that case, the theorem implies that there is a cuspidal eigenform with coefficients in $\Z/p^2\Z$ with reducible pseudorepresentation. If there is a unique cuspidal eigenform with coefficients in $\Z_p$ (i.e.~if $\bT^0$ is smooth over $\Z_p$), then this gives an explicit formula for its reduction modulo $p^2$. In other words, it gives an explicit formula for the extra reducibility.

\subsubsection{Criteria for smoothness of $\bT^0$} 
For $k=2$, Mazur proved a criterion for $\bT^0$ to equal $\Z_p$ in terms of the Weil pairing on $J_0(N)$ \cite[Proposition II.19.2, pg.~140]{mazur1978}. Merel \cite[Th\'eor\`eme 2]{merel1996} built on Mazur's result to prove the remarkable formula that $\bT^0 =\Z_p$ if and only if $\sum_{i=1}^{\frac{N-1}{2}} i\log_N(i)  \equiv0 \pmod{p}$ (this quantity is now called \emph{Merel's number}). Later, Lecouturier \cite[Proposition 1.2]{lecouturier2018} showed that Merel's number vanishes if and only if $\xi_\mathrm{MT}'$ (for $k=2$) does.  Putting these things together, we see that, for $k=2$, we have $\bT^0=\Z_p$ if and only if $\xi_\mathrm{MT}'=0$.

Using the description of $\bT^{0,\red}$ from Theorem \ref{thm:intro T0red}, 
we prove a weight-$k$ analog of this result. When $\bT^0=\Z_p$, we also give an explicit description of the map $\bT^0 \onto \Z/p^2\Z$ in terms of $\xi_\mathrm{MT}'$,  describing the extra reducibility in this case (compare \eqref{eq:mod p2} to the formula \eqref{eq:X0(11)} for $X_0(11)$).

\begin{thm}
\label{thm:intro rank 1}
The inclusion $\Z_p \to \bT^0$ is an isomorphism if and only if both of the following conditions hold:
\begin{enumerate}
\item $\xi_\mathrm{MT}' \ne 0$
\item $I^0$ is principal.
\end{enumerate}
Moreover, if $\Z_p \to \bT^0$ is an isomorphism, then the unique homomorphism $\bT^0 \to \Z/p^2\Z$ is given by\footnote{The quantity in large parentheses should be considered in $\Z/p\Z$.  Note that $\xi_\mathrm{MT}' \ne 0$ by (1), and that, although $\xi_\mathrm{MT}'$ depends on the choice of $\log_N$, the ratio $\frac{\log_N(\ell)}{\xi_\mathrm{MT}'} \in \Z/p\Z$ is independent of choices.  }
\begin{equation}
\label{eq:mod p2}
T_\ell \mapsto 1+\ell^{k-1} + p \left(\zeta(1-k)(1-\ell^{k-1}) 
\frac{(1-N^{\frac{k}{2}})}{p} \frac{\log_N(\ell)}{\xi_\mathrm{MT}'} \right) \pmod{p^2}
\end{equation}
for primes $\ell \ne N$.
\end{thm}
In weight $2$, the Eisenstein ideal is always principal, as was proven by Mazur \cite[Proposition II.16.1, pg.~125]{mazur1978}. It is not always principal in weight $k>2$, but it seems that it is principal if and only if
\[
\prod_{i=1}^{p-1} (1-\zeta_p^i)^{i^{2-k}} \ne 0 \text{ in } \F_N^\times \otimes \F_p
\]
where $\zeta_p \in \F_N^\times$ is a primitive $p$-th root of unity. The `if' part follows from \cite{PG5} for $k \equiv 2 \pmod{p-1}$, and it seems that the same method works in general (see Remark \ref{rem:R=T}). See \cite{deo2021} for some results in this direction. 

\subsection{Applications to Iwasawa theory} 
By combining Theorem \ref{thm:intro T0red} and Theorem \ref{thm:intro rank 1}, we can see that, if $I^0$ is principal and $\xi_\mathrm{MT}' \ne 0$, then there is a cuspidal eigenform $f$ with coefficients in $\Z_p$ whose Galois representation, when reduced modulo $p^2$, is reducible. Using Ribet's technique \cite{ribet1976}, we can use this $f$ to construct non-trivial mod-$p^2$ global Galois cohomology classes that are trivial locally at $p$.

To state this result precisely, we require more notation. Let $\kcyc: G_\Q \to \Z_p^\times$ and $\omega: G_\Q \to \F_p^\times$ be the $p$-adic and mod-$p$ cyclotomic characters, respectively. Let $\log_p: (\Z/p^2\Z)^\times \to \F_p$ be $x \mapsto \frac{\omega^{-1}(x)x-1}{p}$.

 Let $c \in H^1(\Z[1/Np],\F_p(1-k))$ be a non-zero class whose restriction to $H^1(\Q_p,\F_p(1-k))$ is zero (this class is unique up to scaling). Since the vector space $H^1(\Q_N,\F_p(1-k))$ is two-dimensional with canonical basis, we can speak of the \emph{slope} of an element in $H^1(\Q_N,\F_p(1-k))$ (see Section \ref{subsec:cup and slope} for more details).
 \begin{thm}
 \label{thm:intro IMC}
 Assume that $\xi_\mathrm{MT}' \ne 0$ and that $I^0$ is principal. Then:
\begin{enumerate}
\item There is a class $\tilde{c} \in H^1(\Z[1/Np],(\Z/p^2\Z)(\chi_\alpha^{-2} \kcyc^{1-k}))$ that lifts $c$ and whose restriction to $H^1(\Q_p,(\Z/p^2\Z)(\chi_\alpha^{-2} \kcyc^{1-k}))$ is zero.
\item The cup product $c \cup \log_p(\chi_\alpha^{-2} \kcyc^{1-k})$ vanishes.
\item The restriction $c|_N \in H^1(\Q_N,\F_p(1-k))$ of $c$ has slope
\[
\frac{k }{(1-k)}\frac{\zeta(1-k)}{\xi_\mathrm{MT}'}.
\]
\end{enumerate}
Here $\chi_\alpha:G_\Q \to (\Z/p^2\Z)^\times$ is $\sigma \mapsto 1+p\alpha\log_N(\sigma)$, where $\alpha =\frac{(1-N^{k/2})\zeta(1-k)}{p\xi_\mathrm{MT}'} \in \Z/p\Z$, and $(\Z/p^2\Z)(\chi_\alpha^{-2} \kcyc^{1-k})$ denotes $\Z/p^2\Z$ with $G_\Q$ acting by $\chi_\alpha^{-2} \kcyc^{1-k}$.
 \end{thm}
In fact, we show in Section \ref{sec:lift cup slope} that (1) implies (2) and (3) without any assumption. Using the assumptions that $\xi_\mathrm{MT}' \ne 0$ and that $I^0$ is principal, the class $\tilde{c}$ in (1) is constructed using the cuspidal eigenform $f$ (see Corollary \ref{cor:mod p2 repn}).

As we explain in Section \ref{sec:lift cup slope} (see Remark \ref{rem:L-invariant}), the slope of $c|_N$ is a tame analog of the algebraic $\mathcal{L}$-invariant that appears in the Gross-Stark conjecture \cite{gross1981,DDP2011}.  Hence one can think of (3) as a kind of tame analog of the Gross-Stark conjecture.

\subsubsection{Kato's main conjecture for tame families} Theorem \ref{thm:intro IMC} provides a link between the Mazur-Tate $\zeta$-function and Galois cohomology. This is reminiscent of the Iwasawa main conjecture, but, whereas the Iwasawa main conjecture deals with $p$-adic families (i.e.~twists by powers of the $p$-adic cyclotomic character), this result has to do with \emph{tame families} (i.e.~twists by characters of $p$-power-order and conductor $N$). 

Kato has formulated a version of the main conjecture that encompasses very general families \cite{kato1993,kato1993a}. We survey this formulation in the special case of tame families in the second part of this paper. A consequence of Kato's conjecture is that $\xi_\mathrm{MT}$ controls the size of the Galois cohomology $H^2(\Z[1/Np],\Lambda(k))$ (where $G_\Q$ acts on $\Lambda$ via the mod-$N$ cyclotomic character $G_\Q \to (\Z/N\Z)^\times$). Using Theorem \ref{thm:intro IMC} and a method of ``changing Selmer conditions", we give a new proof of the following, which is a consequence of Kato's conjecture.
\begin{cor}
Assume that $\xi_\mathrm{MT}' \ne 0$ and that $I^0$ is principal. Then 
\begin{equation}
\label{eq:intro EIMC}
\Ann_{\Lambda}H^2(\Z[1/Np],\Lambda(k))=\xi_\mathrm{MT}\Lambda.
\end{equation}
\end{cor}
We first prove a result about Galois cohomology with different Selmer conditions (Theorem \ref{thm:IMC modular}), and show that this result is equivalent (Theorem \ref{prop:star and non-star IMC}).

The equality \eqref{eq:intro EIMC}  (without any assumption) is a special case of a known result: the Coates--Sinnott conjecture as formulated by Greither--Popescu \cite{GP2015}. The original results of Coates and Sinnott \cite{CS1974} show that $\xi_\mathrm{MT}$ is in the annihilator, and this suffices to prove the equality in this case by a simple argument. Greither--Popescu \cite{GP2015} give a different proof, showing that the result follows from the Iwasawa main conjecture for totally real fields \cite{wiles1990} and the vanishing of $\mu$-invariants \cite{FW1979}. The novelty of our proof is that we construct the required cohomology classes using tame families of modular forms. See the introductory paragraph to Part 2 for further discussion.

\begin{rem}[Comparison with irregular weight 1 forms]
An analogous situation to the one considered in this paper has been studied, to great effect, by Dasgupta and his coauthors -- first with Darmon and Pollack \cite{DDP2011} and more recently with Kakde and Ventullo \cite{DKV2018}. (For a deformation-theoretic perspective on \cite{DDP2011}, see \cite{BDP2021}.) These authors consider $p$-adic families of cuspforms passing through an irregular weight 1 Eisenstein point. There, the analog of $\bT^{0,\red}$ is computed using linear combinations of Eisenstein series, and the derivative of a $p$-adic $L$-function appears as a coefficient in this linear combination.

One key difference between that situation and ours is that, in our case,  the reducibility quotient $\bT^\red$ is \emph{not} the quotient of $\bT$ acting faithfully on Eisenstein series (and this is the meaning of ``extra" in ``extra reducibility" -- there is more reducibility than is explained just by Eisenstein series). Indeed, in our case, the Eisenstein quotient is $\Z_p$. The extra ``deformation Eisenstein series" only appears when we consider torsion coefficients.
\end{rem}
\subsection{Acknowledgments} This work began while I was a Member at the Institute for Advanced Study working with Akshay Venkatesh. All the ideas for the paper were developed during discussions with Venkatesh and I thank him for generously sharing his ideas and encouragement.
I also thank the Institute for providing a wonderful working environment, and especially the crane who took up residence in the Institute pond for reminding me to think about Iwasawa theory \cite[Section III.1.2.8] {kato1993}.

This work grew out of the joint works \cite{PG3,PG4,PG5} with Carl Wang-Erickson, and it is a pleasure to thank him for his many ideas shared during those collaborations. I also thank Emmanuel Lecouturier for discussions regarding his works \cite{lecouturier2018,lecouturier2017arxiv}, and Tony Feng for discussions about Euler systems. I thank Frank Calegari, Barry Mazur, Andrew Ogg, and Ken Ribet for enlightening conversations and correspondence regarding the history of $X_0(11)$. I thank Samit Dasgupta, Barry Mazur, Andreas Nickel, and Carl Wang-Erickson for comments and corrections to an earlier version of this manuscript. I thank the anonymous referees for their careful reading and thoughtful suggestions which have led to improvements to the exposition.

Finally, it is my great pleasure to thank Kazuya Kato, who introduced me to Iwasawa theory. His works \cite{BK1990, kato1993, kato1993a} are a constant source of inspiration.

This work was supported by the National Science Foundation under the grants DMS-1638352 and DMS-1901867.

\subsection{Notation}
\label{subsec:setup}
The general setup throughout the paper is as follows:
\begin{itemize}
\item $k \ge 2$ is an even integer, and $k >2$ in Part 1\footnote{Everything in Part 1 should work for $k=2$ as well, but requires extra delicacy regarding convergence (the Eisenstein series of weight $2$ and level $1$ is non-holomorphic). We consider the case $k=2$ in greater detail in work-in-progress with Lecouturier and Wang-Erickson.},
\item $p$ is a prime such that $\zeta(1-k) \in \Z_{(p)}^\times$,
\item $N$ is a prime with $p \mid (N-1)$,
\item $\nu >0$ is the $p$-adic valuation of $N-1$.
\end{itemize}
As we remarked at the beginning of Section \ref{subsec:intro eis}, the assumption that $\zeta(1-k) \in \Z_{(p)}^\times$ implies that $(p-1) \nmid k$ (so, in particular, $p \ne 2, 3$).

\begin{rem}
In weight $k$, there are Eisenstein congruences whenever at least one of the following occur
\begin{enumerate}
\item $p$ divides $N^k-1$
\item $p$ divides the numerator of $\zeta(1-k)$.
\end{enumerate}
However, extra reducibility will only occur when $p$ divides $N-1$. In order to focus on this phenomenon, we limit our scope to this situation. 

Also note that, with our assumptions on $N$ and $p$, Eisenstein congruences will only occur in the $-1$-eigenspace for the Atkin-Lehner involution $w_N$, so we focus our attention on this eigenspace. See \cite{PG5} for some cases where congruences occur in $+1$-eigenspaces.
\end{rem}

Throughout the paper, we continue to use the notation $\kcyc$ and $\omega$ for the $p$-adic and mod-$p$ cyclotomic characters, respectively, and we let $\log_p: (\Z/p^2\Z)^\times \to \F_p$ be $x \mapsto \frac{\omega^{-1}(x)x-1}{p}$.

Let $v_p(-)$ denote the $p$-adic valuation on $\Q$. Let $G_{\Q,Np}$ denote the Galois group of the maximal extension of $\Q$ that is unramified outside $Np$, and let $G_{\Q_N}, G_{\Q_p} \subset G_{\Q,Np}$ be a choice of decomposition group at $N$ and $p$. Let $I_N \subset G_{\Q_N}$ and $I_p \subset G_{\Q_p}$ denote the inertia groups. Choose an element $\gamma_N \in I_N$ that topologically generates the pro-$p$ quotient.  
Let $\zeta_N^{(p)} \in \Q(\zeta_N)$ be an element such that $\Q(\zeta_N^{(p)})/\Q$ is the maximal pro-$p$ subextension of $\Q(\zeta_N)/\Q$. Then $\gamma_N$ maps to a generator of $\Gal(\Q(\zeta_N^{(p)})/\Q) \cong (\Z/N\Z)^\times \otimes_\Z \Z_p$. This determines an isomorphism
\[
\log_N:(\Z/N\Z)^\times \otimes_\Z \Z_p \isoto \Z/p^\nu\Z.
\]
We abuse notation and also denote by $\gamma_N$ the element $\log_N^{-1}(1) \in (\Z/N\Z)^\times \otimes_\Z \Z_p$, and denote by $\log_N$ the composite character
\[
G_{\Q,Np} \to \Gal(\Q(\zeta_N^{(p)})/\Q) \isoto (\Z/N\Z)^\times \otimes_\Z \Z_p \isoto \Z/p^\nu\Z.
\]

If $K$ is a $\ell$-adic local field and $x\in A^\times$ for some ring $A$, let $\lambda(x):G_K \to A^\times$ denote the unramified character sending the arithmetic Frobenius to $x$.

If $C$ is a cochain complex, we let $Z^i(C)$ denote the $i$-cocycles and $B^i(C)$ denote the $i$-coboundaries. For a complex like $\RG(G,M)$ we denote $Z^i(\RG(G,M))$, $B^i(\RG(G,M))$, and $H^i(\RG(G,M))$ by $Z^i(G,M)$, $B^i(G,M)$ and $ H^i(G,M)$, respectively, and similarly for the related complexes introduced in Appendix \ref{app:galois}. See Appendix \ref{app:galois} for more notation regarding Galois cohomology.

\part{The Eisenstein ideal for weight $k$ forms}
In this part, we prove most of our main results, including Theorem \ref{thm:intro Rred=Tred} (see Theorem \ref{thm:Rred=Tred}) Theorem \ref{thm:intro T0red} (see Theorem \ref{thm:T0red}), and Theorem \ref{thm:intro rank 1} (see Theorem \ref{thm:rank1} and Corollary \ref{cor:tildef}). 

In Section \ref{sec:modforms}, we review the necessary background material on modular forms and Hecke algebras. In Section \ref{sec:def theory}, we review deformation theory of Galois pseudorepresentations as developed in \cite{PG3,PG4,PG5}; this section includes the definition of $R_\Db$ and the calculation of $R_\Db^\red$. 
In Section \ref{sec:derivative Eisenstein series}, we carry out our main construction of ``derivative Eisenstein series", as inspired by \cite{lecouturier2017arxiv}. In Section \ref{sec:R=T}, we prove our main results.
\section{Modular forms and their Galois representations}
\label{sec:modforms}
In this section, we recall some basic facts about modular forms and their Galois representations. All the results from this section are well-known -- some references are \cite{katz1973,DI1943,gross1990,gouvea1988,mazur1978}., We review them here just to fix our notation.

\subsection{Modular forms and Hecke algebras} We recall some basics about algebraic modular forms.

\subsubsection{Modular forms} Let $\Gamma$ be a subgroup of $\SL_2(\Z)$ with $\Gamma(N) \subset \Gamma$ (we will only consider $\Gamma=\Gamma_0(N)$ or $\Gamma_1(N)$). For a $\Z[1/N]$-module $K$, let $M_k(\Gamma,K)$ denote the module of algebraic modular forms of weight $k$ and level $\Gamma$ with coefficients in $K$, as defined by Katz \cite{katz1973}, and let $S_k(\Gamma,K)$ denote the submodule of cusp forms. If $K$ is a flat $\Z[1/N]$-algebra (such as $\Z_p$), these can be defined in terms of classical modular forms with integral $q$-expansion \cite[Section 1.3]{ohta2014}. For $f \in M_k(\Gamma,K)$ we write its $q$-expansion (at the $\infty$-cusp) as $f(q)= \sum_{n=0}^\infty a_n(f)q^n \in K \otimes \Z[1/N]\lb q \rb$.

\subsubsection{Hecke algebra} Let $\bT'$ denote the sub $\Z_p$-algebra of $\End_{\Z_p}M_k(\Gamma_0(N),\Z_p)$ generated by the $T_\ell$ Hecke operators for primes $\ell \ne N$ together with the Atkin-Lehner involution $w_N$; it is a reduced commutative ring. Let $\bT'^0$ denote the image of $\bT'$ in $\End_{\Z_p}S_k(\Gamma_0(N),\Z_p)$. For a $\bT'$-module $M$ let $M^\pm$ denote the largest direct summand of $M$ on which $w_N$ acts by $\pm 1$.

\subsubsection{Residue exact sequence}There is an exact sequence of $\bT'$-modules
\begin{equation}
\label{eq:res sequence}
0 \to S_k(\Gamma_0(N),\Z_p)^\pm \to M_k(\Gamma_0(N),\Z_p)^\pm \xrightarrow{a_0} \Z_p \to 0
\end{equation}
where $a_0$ is the map sending $f$ to its constant Fourier coefficient $a_0(f)$. The exactness in the middle comes from the fact that $w_N$ switches the two cusps, so a $w_N$-eigenform whose constant term at one cusp is zero automatically has constant term zero at the other cusp. 

The surjectivity of $a_0$ follows from the vanishing of $H^1$ of the sheaf of cusp forms of weight $k$, as in the proof of the base-change property \cite[Theorem 1.7.1]{katz1973}. The surjectivity can also be proven directly from the base-change property, as we now sketch. Suppose, for the sake of contradiction, that the image of $a_0$ is $p^i\Z_p$ for some $i>0$, and let $f \in M_k(\Gamma_0(N),\Z_p)^\pm$ be such that $a_0(f)=p^i$. Then $\bar f :=f \pmod{p}$ is in $S_k(\Gamma_0(N),\F_p)^\pm$ because $a_0(\bar{f})=0$. By the base-change property, there is an $\tilde{f} \in S_k(\Gamma_0(N),\Z_p)^\pm$ with $\tilde{f}\equiv  \bar f\pmod{p}$. Then, since $f -\tilde{f} \equiv 0 \pmod{p}$,  we see that $g:= \frac{\tilde{f}-f}{p}$  is in $M_k(\Gamma_0(N),\Z_p)^\pm$ by the $q$-expansion principle. But $a_0(g)= \frac{a_0(f)}{p} = p^{i-1}$,  contradicting our assumption about the image of $a_0$.

\subsubsection{Duality}
\label{subsub:duality}
There are perfect pairings of free $\Z_p$-modules
\[
\bT'^\pm \times M_k(\Gamma_0(N),\Z_p)^\pm \to \Z_p, \ {\bT'^0}^\pm \times S_k(\Gamma_0(N),\Z_p)^\pm \to \Z_p
\]
given by $(t,f) \mapsto a_1(tf)$. This can be proven just as in \cite[Corollary 2.4.7]{ohta2014}, using integral Atkin-Lehner theory \cite[Proposition 2.1.2]{ohta2014}.

In particular, there is a unique element $T_0 \in \bT'^\pm$ such that $a_1(T_0f)=a_0(f)$ for all $f \in M_k(\Gamma_0(N),\Z_p)^\pm$, which we call the \emph{universal constant term} operator, following Emerton \cite[Section 2]{emerton1999}. Taking the dual of the sequence \eqref{eq:res sequence}, we see that $T_0$ generates the ideal $\ker(\bT'^\pm \to {\bT'^0}^\pm)$ and that this ideal is free of rank 1 as a $\Z_p$-module.

\subsubsection{Eisenstein series of level $\Gamma_0(N)$} Let $E_k(z)$ denote the normalized Eisenstein series of weight $k$ and level 1. It has constant term $\frac{\zeta(1-k)}{2}$ and is an eigenform with $T_\ell$-eigenvalue $1+\ell^{k-1}$ for any prime $\ell$.

Define Eisenstein series $E_{k,N}^\pm$ of level $\Gamma_0(N)$ by
\[
E_{k,N}^\pm(z) = E_k(z) \pm N^{k/2} E_k(Nz).
\]
They are eigenforms with $T_\ell$-eigenvalue $1+\ell^{k-1}$ for any prime $\ell \ne N$ and with $w_N E_{k,N}^\pm = \pm E_{k,N}^\pm$. The have constant terms $a_0(E_{k,N}^\pm)=\frac{1}{2}\zeta(1-k)(1\pm N^{k/2})$.

\subsubsection{Eisenstein series of level $\Gamma_1(N)$} For each non-trivial even character $\chi: (\Z/N\Z)^\times \to \bar{\Q}^\times$, there are two normalized Eisenstein series $E_k(1,\chi)$ and $E_k(\chi,1)$ of level $\Gamma_1(N)$, given by the $q$-expansions
\[
E_k(1,\chi) = \frac{L(1-k,\chi)}{2} + \sum_{n \ge 1} \left(\sum_{d \mid n} \chi(d)d^{k-1} \right) q^n
\]
and
\[
E_k(\chi,1) =\sum_{n \ge 1} \left(\sum_{d \mid n} \chi(n/d)d^{k-1} \right) q^n.
\]
They are eigenforms for all the $T_\ell$. An elementary computation shows that
\begin{equation}
\label{eq:w_N on Echi}
w_N E_k(\chi,1) = \frac{\mathfrak{g}(\chi)}{N^{k/2}}E_k(1,\chi^{-1}),
\end{equation}
where $\mathfrak{g}(\chi) = \sum_{a \in (\Z/N\Z)^\times} \chi(a)e^\frac{2\pi i a}{N}$ is the Gauss sum (see \cite[Proposition 1]{weisinger1977}). Note that, since $w_N$ is an involution, this implies that $w_N E_k(1,\chi^{-1}) = \frac{N^{k/2}}{\mathfrak{g}(\chi)}E_k(\chi,1)$.

\subsection{Eisenstein ideal}
From now on, we only consider $-1$-eigenspaces for $w_N$. We define $E_{k,N}:=E_{k,N}^-$, and let $I'=\Ann_{\bT'}(E_{k,N})$ and let $\m'$ denote the maximal ideal of $\bT'$ generated by $I'$ and $p$. We define $\bT$ and $\bT^0$ to be the completion at $\m'$ of $\bT'$ and $\bT'^0$, respectively. Note that, since $p>2$, the local ring $\bT$ cannot contain any non-trivial involution, so $w_N=-1$ in $\bT$.

For a $\bT'$-module $M$, we let $M_\mathrm{Eis}$ denote $M \otimes_{\bT'} \bT$, and we note that $M_\mathrm{Eis}= (M^-)_\mathrm{Eis}$. In particular, we have the exact sequences
\[
0 \to S_k(\Gamma_0(N),\Z_p)_\mathrm{Eis} \to M_k(\Gamma_0(N),\Z_p)_\mathrm{Eis} \xrightarrow{a_0} \Z_p \to 0
\]
and
\begin{equation}
\label{eq:T to T0}
0 \to T_0\Z_p \to \bT \to \bT^0 \to 0,
\end{equation}
that are dual to each other under the perfect pairings
\[
\bT \times M_k(\Gamma_0(N),\Z_p)_\mathrm{Eis} \to \Z_p, \ \bT^0 \times S_k(\Gamma_0(N),\Z_p)_\mathrm{Eis} \to \Z_p.
\]

The normalization of $\bT$ is the product $\prod_f \sO_f$ where $f$ ranges over normalized eigenforms in $M_k(\Gamma_0(N),\bar\Q_p)_\mathrm{Eis}$, and where $\sO_f$ is the valuation ring in the finite extension of $\Q_p$ obtained by adjoining the Fourier coefficients $a_\ell(f)$ for $\ell \ne N$. The normalization map $\bT \to \prod_f \sO_f$ is injective, as $\bT$ is reduced.
\subsection{Representations associated to cusp forms} 
\label{subsub:representations}
Let $f$ be a cuspidal eigenform of level $\Gamma_0(N)$ and weight $k$. Let $\rho_f: G_{\Q,Np} \to \GL_2(\bar{\Q}_p)$ denote the associated $p$-adic Galois representation. It is the unique irreducible representation satisfying $\det(\rho_f)=\kcyc^{k-1}$ and $\tr(\rho_f)(\Fr_\ell)=a_\ell(f)$ for all primes $\ell$ not dividing $Np$.
The following lemma recalls the local properties of this representation in the cases of interest.

\begin{lem}
\label{lem:rep properties}
Assume that $a_p(f)$ is a $p$-adic unit and that $w_N(f)=-f$. Then
\begin{enumerate}
\item The representation $\rho_f|_{G_{\Q_p}}$ is \emph{ordinary}. That is, we have
\[
\rho_f|_{G_{\Q_p}} \sim \ttmat{\kcyc^{k-1}\lambda(a_p(f)^{-1})}{*}{0}{\lambda(a_p(f))}
\]
\item If $f$ is old at $N$, then $\rho_f|_{G_{\Q_N}}$ is unramified.
\item If $f$ is new at $N$, then $\rho_f|_{G_{\Q_N}}$ is \emph{Steinberg}. This is, we have
\[
\rho_f|_{G_{\Q_N}} \sim \lambda(N^{\frac{k}{2}-1})\ttmat{\kcyc}{*}{0}{1}
\]
\end{enumerate}
\end{lem}

\section{Deformation theory}
\label{sec:def theory}
 Let $\Db = \omega^{k-1} \oplus 1:G_{\Q,Np} \to \F_p$, the residual pseudorepresentation of the Eisenstein series of weight $k$. In this section, we define a ring $R_\Db$ which represents the functor for pseudorepresentations $D$ that deform $\Db$ and satisfy certain conditions so they ``look like'' pseudorepresentations associated to modular forms of weight $k$ and level $\Gamma_0(N)$. 

A method for imposing these conditions has developed extensively in the author's previous papers with Carl Wang-Erickson \cite{PG3,PG4,PG5}, specifically in the case of weight $k=2$. In this paper, we simply sketch how the methods of those papers can be adapted to weight $k$. We freely use the language of pseudorepresentations and Cayley-Hamilton representations. We refer the reader to \cite[Section 3]{PG5} for more detail.

\subsection{Deformation ring}
Let $R_\Db$ denote the pseudodeformation ring parameterizing deformations $D:G_{\Q,Np}\to A$ of $\Db$, where $A$ is an Artin local $\Z_p$-algebra with residue field $\F_p$, subject to the following conditions:
\begin{itemize}
\item $\det(D)=\kcyc^{k-1}$,
\item $D$ is ordinary at $p$,
\item $D$ is unramified-or-Steinberg\footnote{This condition would be called ``unramified-or-$(-1)$-Steinberg'' in \cite{PG5}, where the sign refers to a choice of unramified quadratic twist related to the $w_N$-eigenvalue.  In this paper, we only consider a single twist (because we only consider the $w_N=-1$-eigenspace), so we drop the sign from the notation.} at $N$.
\end{itemize}
These latter two conditions need definitions. By definition, they are true if and only if there is a Cayley-Hamilton representation $\rho:G_{\Q,Np} \to E^\times$ inducing $D$ with the same property. We now define these properties for Cayley-Hamilton representations.

A Cayley-Hamilton representation $\rho:G_{\Q,Np} \to E^\times$ is \emph{ordinary at $p$} if 
\[
(\rho(\sigma)-\kcyc^{k-1}(\sigma))(\rho(\tau)-1)=0
\]
for all $\sigma, \tau \in I_p$, the inertia group of $G_{\Q_p}$.

A Cayley-Hamilton representation $\rho:G_{\Q,Np} \to E^\times$ is \emph{unramified-or-Steinberg at $N$} if 
\[
(\rho(\sigma)-\kcyc(\sigma)\lambda(N^{\frac{k}{2}-1})(\sigma))(\rho(\tau)-\lambda(N^{\frac{k}{2}-1})(\tau))=0
\]
for all $(\sigma, \tau) \in I_N \times G_{\Q_N} \cup G_{\Q_N} \times I_N$. Note that an unramified representation will satisfy this property: if $\sigma \in I_N$, then the first factor is zero, and if $\tau \in I_N$, then the second factor is zero.

Given these definitions, the existence of the deformation ring $R_\Db$ parameterizing deformations with these conditions is proven exactly as in \cite[Section 3]{PG5}.  The idea of the construction is to start with the universal Cayley-Hamilton algebra and impose these conditions by taking a quotient in the category of Cayley-Hamilton algebras; the ring $R_\Db$ is obtained as the scalar ring of this quotient.

We let $D^u:G_{\Q,Np} \to R_\Db$ denote the universal pseudorepresentation. The pseudorepresentation $\kcyc^{k-1}\oplus 1$ over $\Z_p$, which we refer to as the \emph{minimal deformation}, defines a map $R_\Db \onto \Z_p$ which gives $R_\Db$ the structure of an augmented $\Z_p$-algebra; we call the kernel $J^{\min} \subset R_\Db$, and refer to it as the \emph{minimality ideal}.

\subsection{Map $R_\Db \to \bT$}
\label{subsec:R to T}
There is a unique surjective $\Z_p$-algebra homomorphism
\[R_\Db \to \bT
\]
such that $\mathrm{trace}(D^u)(\Fr_\ell) \mapsto T_\ell$ for all $\ell \nmid Np$. Using the fact that $R_\Db$ is generated by the elements $\mathrm{trace}(D^u)(\Fr_\ell)$ as a $\Z_p$ algebra, this map can be constructed and proven to be surjective just as in the proof of \cite[Proposition 4.1.1]{PG5}, following three steps:
\begin{description}
\item[{\bf Step 1}] The pseudorepresentation associated to an eigenform $f$ defines a map $R_\Db \to \sO_f$ sending $\mathrm{trace}(D^u)(\Fr_\ell)$ to $a_\ell(f)$ for all $\ell \nmid Np$. (The fact that this pseudorepresentation satisfies the required conditions follows from Lemma \ref{lem:rep properties}.)
\item[\bf Step 2] The resulting map $R_\Db \to \prod_f \sO_f$ sends $\mathrm{trace}(D^u)(\Fr_\ell)$ to the image of $T_\ell$ under the normalization $\bT \to \prod_f \sO_f$. Hence the map $R_\Db \to \prod_f \sO_f$ factors through a map $R_\Db \to \bT$ whose image is the subalgebra generated by the $T_\ell$ for $\ell \nmid Np$.
\item[\bf Step 3] The image of $R_\Db \to \bT$ also contains $T_p$ (using the interpretation of $a_p(f)$ in terms of Galois representations). This completes the proof the $R_\Db \to \bT$ is surjective.
\end{description}

\begin{rem}
\label{rem:R=T}
We expect that the map $R_\Db \onto \bT$ is an isomorphism. This kind of result was proven in \cite{PG3} in the weight 2 case. However, since the there is no ``finite flat" condition in weight $k>2$, this situation is more closely analogous to the case $k=2$ and level $\Gamma_0(Np)$, which was treated in \cite{PG5}. It seems that the same method can prove that if the restriction map 
\begin{equation}
\label{eq:res on Fp(k-1)}
H^1(\Z[1/p],\F_p(k-1)) \to H^1(\Q_N,\F_p(k-1))
\end{equation}
is non-zero, then $J^{\min}$ is principal and $R_\Db \onto \bT$ is an isomorphism. Using our assumption that $\zeta(1-k) \in \Z_{(p)}^\times$, it is easy to see that $H^1(\Z[1/p],\F_p(k-1))$ is generated by image of the Deligne-Soul\'e cyclotomic element \cite{deligne1979,soule1987} (see also \cite[Section 5]{kurihara1992}),  so \eqref{eq:res on Fp(k-1)} is zero if and only if
\begin{equation}
\label{eq:sum of logs}
\sum_{i=1}^{p-1} i^{k-2} \log_N(1-\zeta_p^i) \equiv 0 \pmod{p}
\end{equation}
where $\zeta_p \in \F_N$ is any primitive $p$th root of unity. In the case $k \equiv 2 \pmod{p-1}$, this is equivalent to $\log_N(p) \equiv 0 \pmod{p}$, which is the condition that was considered in \cite{PG5}. See \cite{deo2021} for some recent results regarding this.
\end{rem}

The map $R_\Db \to \bT \to \bT/I \isoto \Z_p$ coincides with the minimal deformation $R_\Db \to \Z_p$, so the map $R_\Db \to \bT$ is a map of augmented $\Z_p$-algebras.

\subsection{The group ring $\Lambda$ and its quotient $\Lambda_1$}
\label{subsec:def of Lambda}
In order to describe the reducible quotient of $R_\Db$, we set up some notation regarding group rings that will also be used later.

We let $\Lambda=\Z_p[\Gal(\Q(\zeta_N^{(p)})/\Q)]$, which we think of as the universal unramified-outside-$N$ deformation ring of the trivial character $G_{\Q,N} \to \F_p^\times$. The localization $\Lambda \otimes \Q$ is a product of totally ramified finite extensions $\Q_p(\chi)$ of $\Q_p$ labeled by characters $\chi:\Gal(\Q(\zeta_N^{(p)})/\Q) \to \bar{\Q}^\times$:
\[
\Lambda \otimes \Q \cong \prod_{\chi} \Q_p(\chi).
\]

 We write $I_\mathrm{Aug} =\ker(\Lambda \to \Z_p)$ for the augmentation ideal, and let $\Lambda_1 = \Lambda/I_\mathrm{Aug}^2$, which we think of as parameterizing first-order deformations. 

The choices made in Section \ref{subsec:setup} define isomorphisms
\begin{equation}
\label{eq:present Lambda}
\Lambda \isoto \Z_p[\Z/p^\nu\Z] \isoto \frac{\Z_p[X]}{((X+1)^{p^\nu}-1)}.
\end{equation}
such that the composition is $[a] \mapsto (1+X)^{\log_N(a)}$ for a group-like element $[a]$. These isomorphisms define an isomorphism
\begin{equation}
\label{eq:present Lambda1}
\Lambda_1 \isoto \frac{\Z_p[X]}{(X^2,p^\nu X)}
\end{equation}
via $[a] \mapsto 1+\log_N(a)X$. Throughout the paper, we will forget about \eqref{eq:present Lambda}, but we will use \eqref{eq:present Lambda1} as an identification. In particular, the letter $X$ will always refer to an element of $\Lambda_1$ that is a generator of $I_\mathrm{Aug}/I_\mathrm{Aug}^2$ inducing the isomorphism $I_\mathrm{Aug}/I_\mathrm{Aug}^2 \cong \Z/p^\nu\Z$ of Section \ref{subsec:setup}.

We let $\bar{\Lambda}_1= \Lambda_1/p^\nu\Lambda_1$. Via the identification \eqref{eq:present Lambda1}, $\bar{\Lambda}_1$ is identified with the ring of dual numbers over $\Z/p^\nu\Z$. The quotient map $\Lambda_1 \to \bar{\Lambda}_1$ induces an isomorphism 
\begin{equation}
\label{eq:XLambda=XbarLambda}
X \Lambda_1 \isoto X \bar\Lambda_1 = X\cdot  \Z/p^\nu\Z,
\end{equation}
which we also use as an identification.

We let $\dia{-}:G_{\Q,Np} \to \Lambda_1^\times$ denote the character $\dia{\sigma}=1+\log_N(\sigma)X$.
\subsection{Reducible deformation ring} Let $R_N \to R_N^\red$ denote the quotient parameterizing deformations $D:G_{\Q,Np} \to A$ that are still reducible (that is $D=\chi_1 \oplus \chi_2$ for some characters $\chi_1,\chi_2: G_{\Q,Np} \to A^\times$).

\begin{lem}
\label{lem:Rred}
The pseudorepresentation 
\[
D^\red:G_{\Q,Np}\to \Lambda_1
\]
determined by $D^\red = \kcyc^{k-1}\dia{-}^{-1} \oplus \dia{-}$ is a deformation of $\Db$ and determines an isomorphism
\[
R_\Db^\red \isoto \Lambda_1.
\]
\end{lem}
\begin{proof}
 It can be checked easily that $D^\red$ defines a reducible deformation of $\Db$ that is ordinary at $p$ and unramified-or-Steinberg at $N$, so it defines a map $R_\Db^\red \to \Lambda_1$. We must construct the inverse map. The proof is just as in \cite[Proposition 5.1.2]{PG3} or \cite[Lemma 4.2.3]{PG5}, so we will simply sketch the argument.

Let $D^{u,\red}=\chi_1 \oplus \chi_2: G_{\Q,Np} \to R_\Db^\red$ denote the universal reducible deformation. Since $\det(D^{u,\red})=\kcyc^{k-1}$ we can write $\chi_1=\kcyc^{k-1}\chi_2^{-1}$.
Now $\chi_2$ is a deformation of the trivial character. The ordinary condition forces $\chi_2$ to be unramified at $p$. Hence there is a surjection $\Lambda \onto R_\Db^\red$. The unramified-or-Steinberg at $N$ condition forces this surjection to factor through
\[
\frac{\Z_p[\Gal(\Q(\zeta_N^{(p)})/\Q)]}{([\gamma_N]+[\gamma_N]^{-1}-2)}.
\]
Under the isomorphism \eqref{eq:present Lambda}, this quotient is identified with $\Lambda_1$ and the universal character is identified with $\dia{-}$. This gives a surjection $\Lambda_1 \to R_\Db^\red$ that is inverse to $R_\Db^\red \to \Lambda_1$.
\end{proof}

\section{Derivative Eisenstein series}
\label{sec:derivative Eisenstein series}
\subsection{Mazur-Tate $\zeta$-function}
\label{subsec:MT} We first consider the constant terms $L(1-k,\chi)$ of the Eisenstein series $E_k(1,\chi)$ and note that they interpolate into an element $\xi_\mathrm{MT} \in \Lambda$ that we call the Mazur-Tate $\zeta$-function (after \cite{MT1987}).\footnote{It could be called the ``tame $L$-function", to highlight the analogy with $p$-adic $L$-functions, or a Stickelberger element.}

We consider the function $\chi \mapsto L(1-k,\chi)$ for a character $\chi:\Gal(\Q(\zeta_N^{(p)})/\Q)  \to \bar\Q_p^\times$,  where $L(s,\chi)$ is the Dirichlet $L$-function and we think of $\chi$ as a Dirichlet character of modulus $N$. A priori, this function is an element $\xi_\mathrm{MT}$ of the group ring $\bar\Q_p[\Gal(\Q(\zeta_N^{(p)})/\Q)]$.  Explicitly, we have the formula (see e.g.~\cite[Thm.~4.2, pg.~32]{washington1997}),
\[
L(1-k,\chi) = -\frac{B_{k,\chi}}{k} = -\frac{N^{k-1}}{k} \sum_{a=1}^{N-1}B_k(a/N)\chi(a)
\]
for any Dirichlet character of modulus $N$, 
where $B_{k,\chi}$ is the Bernoulli number and $B_k(x)$ is the Bernoulli polynomial, so $\xi_\mathrm{MT} = -\frac{N^{k-1}}{k} \sum_{a=1}^{N-1}B_k(a/N)[a]$, and we see $\xi_\mathrm{MT} \in \Q_p[\Gal(\Q(\zeta_N^{(p)})/\Q)]$. It is known that $\xi_\mathrm{MT}$ is the integral subring $\Lambda \subset \Q_p[\Gal(\Q(\zeta_N^{(p)})/\Q)]$ (see \cite[Theorem 1.2]{CS1974}, and note that their $k+1$ is our $k$, and that the integer $w_k(\Q)$ appearing in the statement is prime-to-$p$ because of our assumption $(p-1)\nmid k$).

In the next lemma, we use the identifications of Section \ref{subsec:def of Lambda}.

\begin{lem}
\label{lem: zeta MT}
The image of $\xi_\mathrm{MT}$ in $\Lambda_1$ is given by
\[
\zeta(1-k)(1-N^{k-1}) + \xi_\mathrm{MT}'X
\]
where $\xi_\mathrm{MT}' \in \Z/p^\nu\Z$ is the element
\[
-\frac{N^{k-1}}{k} \sum_{a=1}^{N-1}B_n(a/N)\log_N(a) \in \Z/p^\nu\Z
\]
In particular, the image of $\xi_\mathrm{MT}$ in $\Lambda_1$ annihilates $X$.
\end{lem}
\begin{proof}
The image of $\xi_\mathrm{MT}$ under the augmentation is $L(1-k,\mathbbm{1})$, where $\mathbbm{1}$ denotes the trivial character modulo $N$. This is easily seen to be equal to $\zeta(1-k)(1-N^{k-1})$. 

This shows that 
\[
\xi_\mathrm{MT}-\zeta(1-k)(1-N^{k-1})=-\frac{N^{k-1}}{k}\sum_{a=1}^{N-1}B_k(a/N)([a]-1).
\]
The isomorphism $I_\mathrm{Aug}/I_\mathrm{Aug}^2 \cong \Z/p^\nu\Z$ sends $[a]-1$ to $\log_N(a)$ and sends $X$ to $1$, so we have
\[
\xi_\mathrm{MT}-\zeta(1-k)(1-N^{k-1}) = \left(-\frac{N^{k-1}}{k}\sum_{a=1}^{N-1}B_k(a/N)\log_N(a)\right)X.
\]
The last statement follows from the fact that the annihilator of $X$ in $\Lambda_1$ is the ideal generated by $N-1$ and $X$.
\end{proof}
\subsection{Group-ring valued Eisenstein series}
Consider the Eisenstein series
\[
E_k(1,[-]) = \frac{1}{2}\xi_\mathrm{MT} + \sum_{n \ge 1} \left(\sum_{d \mid n} [d]d^{k-1} \right) q^n \in \Z_p[(\Z/N\Z)^\times/\{\pm\}]\lb q \rb
\]
and
\[
E_k([-],1) =\sum_{n \ge 1} \left(\sum_{d \mid n} [n/d]d^{k-1} \right) q^n \in \Z_p[(\Z/N\Z)^\times/\{\pm\}]\lb q \rb.
\]

The ring $\bar\Q_p[(\Z/N\Z)^\times/\{\pm\}]$ is a product of $\bar\Q_p$ labeled by even characters $\chi$; the map associated to a given $\chi$ sends $E_k(1,[-])$ to $E_k(1,\chi)$ and $E_k([-],1)$ to $E_k(\chi,1)$. This implies that these $q$-series are $q$-expansions of modular forms elements of $M_k(\Gamma_1(N),\bar\Q_p[(\Z/N\Z)^\times/\{\pm\}])$. By the $q$-expansion principle \cite[Corollary 1.6.2]{katz1973}, they are actually elements of $M_k(\Gamma_1(N),\Z_p[(\Z/N\Z)^\times/\{\pm\}])$.
\subsection{Derivative Eisenstein series} Throughout the rest of this section, we frequently use the notation for the group ring $\Lambda$, its quotients $\Lambda_1$ and $\bar\Lambda_1$, and the element $X \in \Lambda_1$ introduced in Section \ref{subsec:def of Lambda}.

Let $E_k(1,\dia{-}), E_k(\dia{-},1) \in M_k(\Gamma_1(N),\Lambda_1)$ be the base-change of $E_k(1,[-])$ and $E_k([-],1)$ via the quotient $ \Z_p[(\Z/N\Z)^\times/\{\pm\}]\to \Lambda_1$. Let $\bar{E}_k(1,\dia{-}),\bar{E}_k(\dia{-},1) \in M_k(\Gamma_1(N),\bar\Lambda_1)$ be further the base change via $\Lambda_1 \onto \bar\Lambda_1$.

By base-changing along the inclusion $\Z_p \subset \Lambda_1$, we can consider $E_{k,N}$ as an element in $M_k(\Gamma_1(N),\Lambda_1)$. Then we have 
\[
X \cdot E_{k,N} \in M_k(\Gamma_1(N),X\Lambda_1)=M_k(\Gamma_1(N),X\bar\Lambda_1),
\]
where we have used the identification \eqref{eq:XLambda=XbarLambda}. We also have 
\[X \cdot \bar{E}_k(1,\dia{-}), \ X \cdot \bar{E}_k(\dia{-},1) \in M_k(\Gamma_1(N),X\bar\Lambda_1).
\]

\begin{lem}
\label{lem:XE}
We have
\[
X \cdot E_{k,N} = X \cdot \bar{E}_k(1,\dia{-}) = X \cdot \bar{E}_k(\dia{-},1)
\]
as elements of $M_k(\Gamma_1(N),X\bar\Lambda_1)$.
\end{lem}
\begin{proof}
In this proof, we frequently use the fact that $N=1$ in $\bar\Lambda_1$. By the $q$-expansion principle, we need only check that these forms have the same $q$-expansion. We first check the constant terms. We have $a_0(E_{k,N})=\frac{1}{2}\zeta(1-k)(1-N^{k/2})$, which is zero in $\bar\Lambda_1$, so $a_0(X \cdot E_{k,N})=0$. We also see trivially that $a_0(\bar{E}_k(\dia{-},1))=0$, so we have to check that $a_0(X \cdot \bar{E}_k(1,\dia{-}))=0$. But we have $a_0(X \cdot E_k(1,\dia{-}))= \frac{1}{2} X \xi_\mathrm{MT}$, which is zero in $\bar\Lambda_1$ by Lemma \ref{lem: zeta MT}. 

Next we consider $a_N$ coefficients. We have 
\begin{eqnarray*}
& a_N(X \cdot E_{k,N}) = X(1+N^{k-1}-N^{k/2}) = X \\
& a_N(X \cdot  \bar{E}_k(1,\dia{-})) = X \\
& a_N(X \cdot  \bar{E}_k(\dia{-},1))= N^{k-1}X = X.
\end{eqnarray*}
Finally, we check easily that, for any prime $\ell \ne N$, all three have $a_\ell$-coefficient $X(1+\ell^{k-1})$.
\end{proof}

Let $E'_{k,N} \in M_k(\Gamma_1(N),\bar\Lambda_1)$ be the element
\[
E'_{k,N} = \bar{E}_k(1,\dia{-})-\bar{E}_k(\dia{-},1),
\]
which we call the \emph{derivative Eisenstein series}. 

\begin{lem}
\label{lem:E'}
The $q$-expansion of derivative Eisenstein series $E'_{k,N}$ takes values in $X \bar\Lambda_1$. Moreover,
\begin{enumerate}
\item the diamond operators act trivially on $E'_{k,N}$, 
\item $a_0(E'_{k,N})=\frac{1}{2}\xi'_\mathrm{MT}X$,
\item $E'_{k,N}|_{(T_\ell-\ell^{k-1}-1)} = \log_N(\ell)(\ell^{k-1}-1)X \cdot E_{k,N}$ for any prime $\ell \ne N$, and
\item $E'_{k,N}|_{w_N} = -E'_{k,N}$.
\end{enumerate}
In particular, $E'_{k,N} \in M_k(\Gamma_0(N),X\bar\Lambda_1)_\mathrm{Eis}$.
\end{lem}
\begin{proof}
It follows from Lemma \ref{lem:XE} that $X \cdot E'_{k,N}=0$. This implies that the $q$-expansion of $E'_{k,N}$ takes values in the annihilator of $X$ in $\bar\Lambda_1$, which is $X \bar\Lambda_1$. This proves the first statement. We proceed with the numbered statements.

(1) Since $\gamma_N$ is a generator of $(\Z/N\Z)^\times \otimes \Z_p$, it's enough to show that $\dia{\gamma_N}-1$ acts by zero. But $\dia{\gamma_N}-1$ acts as multiplication by $X$, which annihilates $E'_{k,N}$ since it has coefficients in $X\bar\Lambda_1$.

(2) Since $a_0(\bar{E}_k(\dia{-},1))=0$, we have $a_0(E'_{k,N}) = a_0(\bar{E}_k(1,\dia{-}))$, which is the image of $\frac{1}{2}\xi_\mathrm{MT}$ in $\bar\Lambda_1$. This is equal to $\frac{1}{2}\xi_\mathrm{MT}'X$ by Lemma \ref{lem: zeta MT}.

(3) Using the fact that $ \bar{E}_k(1,\dia{-})$ and $\bar{E}_k(\dia{-},1)$ are eigenforms for $T_\ell$, we easily compute that
\begin{eqnarray*}
 \bar{E}_k(1,\dia{-})|_{(T_\ell-\ell^{k-1}-1)} = \ell^{k-1} \log_N(\ell) X \cdot \bar{E}_k(1,\dia{-}) \\
 \bar{E}_k(\dia{-},1)|_{(T_\ell-\ell^{k-1}-1)} = \log_N(\ell) X \cdot \bar{E}_k(\dia{-},1).
\end{eqnarray*}
The result now follows from Lemma \ref{lem:XE}.

(4) By \eqref{eq:w_N on Echi}, we have
\[
E'_{k,N}|_{w_N} = \frac{N^{k/2}}{\mathfrak{g}(\dia{-}^{-1})} \bar{E}_k(\dia{-}^{-1},1) - \frac{\mathfrak{g}(\dia{-})}{N^{k/2}}\bar{E}_k(1,\dia{-}^{-1}),
\]
where $\mathfrak{g}(-)$ denotes the Gauss sum. Now we compute that
\[
\mathfrak{g}(\dia{-}) = \sum_{a=1}^{N-1} (1+\log_N(a)X) \zeta_N^a = -1 + \mathfrak{g}(\log_N) X.
\]
It follows that $\mathfrak{g}(\dia{-}^{-1}) =  \mathfrak{g}(\dia{-})^{-1}$. Using the fact that $N=1$ in $\bar\Lambda_1$, we have
\[
E'|_{w_N} = \mathfrak{g}(\dia{-}) ( \bar{E}_k(\dia{-}^{-1},1) - \bar{E}_k(1,\dia{-}^{-1})).
\]

\begin{claim} We have $\bar{E}_k(\dia{-}^{-1},1) - \bar{E}_k(1,\dia{-}^{-1}) = E_{k,N}'.$
\end{claim}
\begin{proof}
The automorphism $\iota: [g] \mapsto [g^{-1}]$ of $\bar\Lambda_1$ (thought of as quotient of the group ring) acts by $-1$ on $X \cdot \bar\Lambda_1$. This implies that $\iota(E')$ equals $-E'$ on one hand, and equals $\bar{E}_k(1,\dia{-}^{-1})-\bar{E}_k(\dia{-}^{-1},1)$ on the other hand.
\end{proof}
Hence we have
\[
E'_{k,N}|_{w_N} = \mathfrak{g}(\dia{-}) ( \bar{E}_k(\dia{-}^{-1},1) - \bar{E}_k(1,\dia{-}^{-1})) = \mathfrak{g}(\dia{-}) E_{k,N}'.
\]
Since $\mathfrak{g}(\dia{-}) \equiv -1 \pmod{X}$, and since $E'_{k,N}$ has coefficients in $X \bar\Lambda_1$, we have $\mathfrak{g}(\dia{-}) E_{k,N}' = - E'_{k,N}$. This completes the proof of (4).

To see the final statement, note that, for any $\Z[1/N]$-module $A$, the module $M_k(\Gamma_0(N),A)$ is the invariants of $M_k(\Gamma_1(N),A)$ under the diamond operators (see \cite[Lemma 1.2.6]{ohta2014} or a similar argument in \cite[Section 2.1]{edixhoven1992}). So by (1), we have $E_{k,N}'\in M_k(\Gamma_0(N),X \bar\Lambda_1)$. Parts (3) and (4) show that $I^2E_{k,N}'=0$, so $E_{k,N}'\in M_k(\Gamma_0(N),X \bar\Lambda_1)_\mathrm{Eis}$.
\end{proof}

\subsection{Deformation Eisenstein series} 
Consider the modular form 
\[
\tilde{E}_{k,N} = E_{k,N} + E'_{k,N},
\]
where the sum is taking place in $M_k(\Gamma_0(N),\Lambda_1)_\mathrm{Eis}$. We have
\begin{equation}
\label{eq:def of xiEis}
a_0(\tilde{E}_{k,N})=\frac{1}{2}\left(\zeta(1-k)(1-N^{k/2}) + \xi_\mathrm{MT}'X\right).
\end{equation}

We define $\xi_\mathrm{MT}^\mathrm{Eis}:=a_0(\tilde{E}_{k,N}) \in \Lambda_1$. Although $\xi_\mathrm{MT}^\mathrm{Eis} \ne \xi_\mathrm{MT}$, we  use this notation to invoke the idea of $\xi_\mathrm{MT}^\mathrm{Eis}$ as an altered version of $\xi_\mathrm{MT}$:  the incarnation of $\xi_\mathrm{MT}$ as the constant term of an Eisenstein series.
 Note that $X\xi_\mathrm{MT}^\mathrm{Eis}=0$, just as in Lemma \ref{lem: zeta MT}.

\begin{prop}
\label{prop:Derivative Eisenstein series}
There is a surjective morphism of augmented $\Z_p$-algebras 
\[\bT \onto \Lambda_1\]
defined by $T_\ell \mapsto a_\ell(\tilde{E}_{k,N})=1+\ell^{k-1} + (1-\ell^{k-1})\log_N(\ell)X$ for all $\ell \nmid N$. This map sends $T_0$ to $\xi_\mathrm{MT}^\mathrm{Eis}$.
\end{prop}
\begin{proof}
By duality (Section \ref{subsub:duality}), we have to show that $\tilde{E}_{k,N}$ is annihilated by the following Hecke operators:
\begin{enumerate}
\item $T_\ell - (1+\ell^{k-1} + \log_N(\ell)(\ell^{k-1}-1)X)$
\item $w_N+1$.
\end{enumerate}
This follows from Lemma \ref{lem:E'}.
\end{proof}

\section{An ``$R^\red=\bT^\red$" theorem}
\label{sec:R=T}

\subsection{Reducible modularity}
Let $J^\red=\ker(R_\Db \to R_\Db^\red)$ and let $I^\red \subset \bT$ be the image of $J^\red$ under $R_\Db \to \bT$. Let $\bT^\red = \bT/I^\red$. Recall the surjective homomorphism $R_\Db \onto \bT$ of augmented $\Z_p$-algebras defined in Section \ref{subsec:R to T}.

\begin{thm}
\label{thm:Rred=Tred}
The map $R_\Db \onto \bT$ induces an isomorphism $R_\Db^\red\to \bT^\red$. The inverse map is composite of map $\bT^\red \to \Lambda_1$ induced by Proposition \ref{prop:Derivative Eisenstein series} and the isomorphism $\Lambda_1 \to R_\Db^\red$ of Lemma \ref{lem:Rred}.
\end{thm}
\begin{proof}
The isomorphism $R_\Db^\red \isoto \Lambda_1$ of Lemma \ref{lem:Rred} sends $\mathrm{trace}(D^u)(\Fr_\ell)$ for $\ell \ne N$ to 
\begin{align*}
\kcyc(\Fr_\ell)^{k-1}\dia{\Fr_\ell}^{-1} + \dia{\Fr_\ell} & = \ell^{k-1}(1+\log_N(\ell)X) + 1-\log_N(\ell)X \\
& = 1+\ell^{k-1} + (1-\ell^{k-1})\log_N(\ell)X,
\end{align*}
which is the image of $T_\ell$ under the map $\bT \to \Lambda_1$ of Proposition \ref{prop:Derivative Eisenstein series}. Since $R_\Db$ is generated by the elements $\mathrm{trace}(D^u)(\Fr_\ell)$, this implies that the two composite maps
\[
R_\Db \onto R_\Db^\red \isoto \Lambda_1, \ R_\Db \onto \bT \onto \Lambda_1
\]
coincide. Hence the latter map sends $J^\red$ to zero, and the induced composite map
\[
R_\Db^\red \onto \bT^\red \to \Lambda_1
\]
is the isomorphism of Lemma \ref{lem:Rred}. This implies that $R_\Db^\red \onto \bT^\red$ is injective, and hence an isomorphism.
\end{proof}
Let $I^{0,\red} \subset \bT^0$ be the image of $I^\red$ under $\bT \onto \bT^0$, and let $\bT^{0,\red}:=\bT^0/I^{0,\red}$.

\begin{thm}
\label{thm:T0red}
The isomorphism $\bT^\red \isoto \Lambda_1$ of Theorem \ref{thm:Rred=Tred} induces an isomorphism
\[
\bT^{0,\red} \cong \Lambda_1/\xi_\mathrm{MT}^\mathrm{Eis}.
\]
\end{thm}
\begin{proof}
Note that the map $\bT \onto \Lambda_1$ sends $T_0$ to $\xi_\mathrm{MT}^\mathrm{Eis}$. By the sequence \eqref{eq:T to T0}, this map induces a  map $\bT^{0} \onto \Lambda_1/\xi_\mathrm{MT}^\mathrm{Eis}$. Call the kernel of this map $I'$.  The content of the theorem is that $I'=I^{0,\red}$. 

We have a commutative diagram with exact rows and columns
\[\xymatrix{
		& 						& 0 \ar[d] & 0 \ar[d] \\
		& 						& \Z_p  \ar[d] \ar[r]^-\sim & \xi_\mathrm{MT}^\mathrm{Eis}\Lambda_1  \ar[d] \\
0 \ar[r] & I^\red \ar[r] \ar[d]_-\wr & \bT \ar[r] \ar[d] & \Lambda_1 \ar[d] \ar[r] & 0 \\
0 \ar[r] & I' \ar[r] & \bT^0 \ar[r] \ar[d] &  \Lambda_1/\xi_\mathrm{MT}^\mathrm{Eis}\ar[d] \ar[r] & 0 \\
		& 						& 0  & 0
}\]
where the vertical map $\Z_p \to \bT$ is $1 \mapsto T_0$.  We will show that the map $\Z_p \to  \xi_\mathrm{MT}^\mathrm{Eis}\Lambda_1$ is an isomorphism, which, by the snake lemma, will imply that the map $I^\red \to I'$ is an isomorphism. In other words, this will show $I' = I^{0,\red}$, completing to proof of the theorem.

It remains to show that the map $\Z_p \to  \xi_\mathrm{MT}^\mathrm{Eis}\Lambda_1$ is an isomorphism.  To see this, first note that $\xi_\mathrm{MT}^\mathrm{Eis}\Lambda_1$ is the free $\Z_p$-submodule of $\Lambda_1$ generated by $\xi_\mathrm{MT}^\mathrm{Eis}$. Indeed, $X\xi_\mathrm{MT}^\mathrm{Eis}=0$, so $\xi_\mathrm{MT}^\mathrm{Eis}\Lambda_1=\xi_\mathrm{MT}^\mathrm{Eis}\Z_p$, and, since $\xi_\mathrm{MT}^\mathrm{Eis} \pmod{X\Lambda_1}$ is a non-zero element of $\Z_p$, the module $\xi_\mathrm{MT}^\mathrm{Eis}\Z_p$ is $\Z_p$-torsion-free. Since the map $\Z_p \to \xi_\mathrm{MT}^\mathrm{Eis}\Lambda_1$ sends $1$ to $a_1(T_0\tilde{E}_{k,N})=a_0(\tilde{E}_{k,N})=\xi_\mathrm{MT}^\mathrm{Eis}$, we see that this map is an isomorphism. 
\end{proof}
\begin{rem}
Note that this theorem implies the equality
\begin{equation}
\label{eq:T0 mod I0}
\bT^0/I^0=\Z_p/a_0(E_{k,N})\Z_p=\Z/p^{\nu+v_p(k)}\Z,
\end{equation}
which is reminiscent of Mazur's result \cite[Proposition II.9.6, pg.~96]{mazur1978} on the index of the Eisenstein ideal. The theorem itself is reminiscent of results of Wiles and Mazur-Wiles (for example \cite[Theorem 4.1]{wiles1990}) relating the intersection between Eisenstein and cuspidal Hida families to the Kubota-Leopoldt $p$-adic $L$-function. The idea to prove this kind of result using the universal constant term operator originated with Emerton \cite{emerton1999}.
\end{rem}

\begin{cor}
\label{cor:T0red finite}
The ring $\bT^{0,\red}$ is annihilated by $p^{2\nu+v_p(k)}$. Hence it is the quotient of
\[
\frac{(\Z/p^{2\nu+v_p(k)}\Z)[X]}{(p^\nu X, X^2)}
\]
by the image of $\xi_\mathrm{MT}^\mathrm{Eis}$. In particular, $\bT^{0,\red}$ has finite cardinality.
\end{cor}
\begin{proof}
By Theorem \ref{thm:T0red}, we have $\xi_\mathrm{MT}^\mathrm{Eis}=0$ in $\bT^{0,\red}$. This implies that we have
\[
\zeta(1-k)(1-N^{k/2}) +\xi_\mathrm{MT}'X =0
\]
as elements on $\bT^{0,\red}$. Since $v_p(\zeta(1-k)(1-N^{k/2}))=\nu+v_p(k)$, we see that $p^{\nu+v_p(k)}$ is in the ideal $X \cdot \bT^{0,\red}$. Since $p^\nu X =0$ in $\Lambda_1$, we see that $p^{2\nu+v_p(k)}=0$ in $\bT^{0,\red}$.
\end{proof}
\subsection{Consequences for modular forms when $\xi_\mathrm{MT}'$ is a unit} 
Since $\xi_\mathrm{MT}'$ is the coefficient of $X$ in $\xi_\mathrm{MT}^\mathrm{Eis}$, if $\xi_\mathrm{MT}'$ is a unit,  then $X$ is equivalent to the image of an element of $\Z_p$ in  $\Lambda_1/\xi_\mathrm{MT}^\mathrm{Eis}$. We introduce a constant to keep track of this element of $\Z_p$.

\begin{defn}
\label{defn:alpha} Suppose that $\xi_\mathrm{MT}' \in \Z/p^\nu\Z$ is a unit. Define the \emph{extra reducibility constant} $\alpha \in \Z/p^\nu\Z$ by
\begin{equation}
\label{eq:alpha}
\alpha =  \xi_\mathrm{MT}'^{-1}\zeta(1-k) \frac{(1-N^{k/2})}{p^{\nu+v_p(k)}}.
\end{equation}
Define the \emph{extra reducibility character} $\chi_\alpha: G_{\Q,Np} \to (\Z/p^{2\nu+v_p(k)}\Z)^\times$ by
the formula 
\begin{equation}
\label{eq:chialpha}
\chi_\alpha(\sigma)=1+p^{\nu+v_p(k)}\alpha \log_N(\sigma).
\end{equation}
\end{defn}
The purpose of this definition comes from the following lemma.

\begin{lem}
\label{lem:T0red when xi' is a unit}
Suppose that $\xi_\mathrm{MT}' \in \Z/p^\nu\Z$ is a unit. Then there are isomorphisms 
\[
\bT^{0,\red} \isoto\frac{(\Z/p^{2\nu+v_p(k)}\Z)[X]}{(p^\nu X, X^2, \xi_\mathrm{MT}'\cdot(p^{\nu+v_p(k)}\alpha -X))} \isoto \Z/p^{2\nu+v_p(k)}\Z
\]
where the first map is $T_\ell \mapsto  1+\ell^{k-1}+(1-\ell^{k-1})\log_N(\ell)X$ and the second map is $X \mapsto p^{\nu + v_p(k)}\alpha$,  and where $\alpha$ is the extra reducibility constant \eqref{eq:alpha}.
\end{lem}
\begin{proof}
By the definition of $\alpha$, we have
\[
\xi_\mathrm{MT}^\mathrm{Eis} \equiv \frac{1}{2} \xi_\mathrm{MT}'\cdot (p^{\nu+v_p(k)}\alpha -X) \pmod{p^{2\nu+v_p(k)},p^\nu X, X^2}.
\]
Hence the lemma follows from Theorem \ref{thm:T0red} and Corollary \ref{cor:T0red finite}.
\end{proof}
We can interpret this in terms of modular forms.
\begin{cor}
\label{cor:z mod p^2nu}
Suppose that $\xi_\mathrm{MT}' \in \Z/p^\nu\Z$ is a unit. Then there is a normalized eigenform $f \in S_k(\Gamma_0(N),\Z/p^{2\nu+v_p(k)}\Z)_\mathrm{Eis}$ with 
\[
a_\ell(f) = 1+\ell^{k-1}+p^{\nu+v_p(k)}(1-\ell^{k-1})\log_N(\ell)\alpha
\]
for all $\ell \ne N$, where $\alpha$ is the extra reducibility constant \eqref{eq:alpha}.
\end{cor}
\begin{proof}
It is equivalent to show that there is a surjective $\Z_p$-algebra homomorphism
\begin{equation}
\label{eq:T0 to T0red}
\bT^0 \onto \Z/p^{2\nu+v_p(k)}\Z.
\end{equation}
sending $T_\ell$ to $1+\ell^{k-1}+p^{\nu+v_p(k)}(1-\ell^{k-1})\log_N(\ell)\alpha$.  This is immediate from Lemma \ref{lem:T0red when xi' is a unit}.
\end{proof}

By the base-change property for algebraic modular forms, there is a cuspform $\tilde{f}\in S_k(\Gamma_0(N),\Z_p)_\mathrm{Eis}$ lifting the eigenform $f$ of the corollary, but there may not be an \emph{eigenform} $\tilde{f}$ lifting $f$ in general. However, if $S_k(\Gamma_0(N),\Z_p)_\mathrm{Eis}$ happens to be rank one as a $\Z_p$-module, then any normalized form is an eigenform, and this guarantees that there is an eigenform $\tilde{f}$ lifting $f$. The next theorem gives a criterion for the rank to be one.

\begin{thm}
\label{thm:rank1} The inclusion $\Z_p \to \bT^0$ is an isomorphism if and only if both of the following conditions hold:
\begin{enumerate}
\item $\xi_\mathrm{MT}'$ is a unit in $\Z/p^\nu\Z$, and
\item $I^0$ is principal.
\end{enumerate}
\end{thm}
\begin{proof}
We first prove the direct implication. If $\bT^0=\Z_p$, then any ideal is principal, so (2) is immediate. On the other hand, if $\xi_\mathrm{MT}'$ is not a unit, then $\xi_\mathrm{MT}^\mathrm{Eis} \equiv 0 \pmod{p}$. Then, using Theorem \ref{thm:T0red}, we have
\[
\bT^{0,\red}/p\bT^{0,\red} \cong \frac{\F_p[X]}{(X^2)}.
\]
Since this is not a quotient of $\Z_p$, we have $\bT^0 \ne \Z_p$. This shows that if $\bT^0=\Z_p$, then (1) is true.

Now assume (1) and (2). Then $I^0$ is principal and we know that $\bT^0/I^0=\Z/p^{\nu+v_p(k)}\Z$ by \eqref{eq:T0 mod I0}. To illustrate the rest of the proof, we first consider the case $\nu=1$ and $v_p(k)=0$. In that case, we see that $\bT^0$ is a DVR with residue field $\F_p$. But by \eqref{eq:T0 to T0red}, we see that $\bT^0$ has $\Z/p^2\Z$ as a quotient; this cannot occur if $\bT^0$ is a ramified DVR, so we must have $\bT^0=\Z_p$.

In the general case, $\bT^0$ need not be a DVR, but we there is a presentation of $\bT^0$ of the form $\frac{\Z_p[t]}{(F(t))} \isoto \bT^0$, where $t$ maps to a generator of $I^0$ and $F(t)$ is the characteristic polynomial of $t$ acting on $\bT^0$. By \eqref{eq:T0 mod I0}, we have $F(0)=u p^{\nu+v_p(k)}$ with $u \in \Z_p^\times$, and, since $\bT^0$ is local, $F(t)$ is a distinguished polynomial. Assume, for a contradiction, that $\deg(F)>1$. In that case $F(t) \equiv a_1t+u p^{\nu+v_p(k)} \pmod{t^2}$ with $a_1 \in p\Z_p$.

Composing our presentation with with the map \eqref{eq:T0 to T0red} from Corollary \ref{cor:z mod p^2nu} we obtain a map $\phi:\frac{\Z_p[t]}{(F(t))} \onto \Z/p^{2\nu+v_p(k)}\Z$. By \eqref{eq:T0 mod I0}, we must have $\phi(t)=vp^{\nu+v_p(k)}$ for some $v \in \Z_p^\times$; in particular, $\phi(t^2)=0$. Then $\phi$ factors through a map
\[
\frac{\Z_p[t]}{(a_1 t+ u p^{\nu+v_p(k)},t^2)} \onto \Z/p^{2\nu+v_p(k)}\Z, \ t \mapsto vp^{\nu+v_p(k)}
\]
and we have
\[
0=a_1vp^{\nu+v_p(k)}+ u p^{\nu+v_p(k)}= (a_1v+u)p^{\nu+v_p(k)}
\]
in $\Z/p^{2\nu+v_p(k)}\Z$. But since we assume $a_1 \in p\Z_p$ and $u \in \Z_p^\times$, we have $a_1v+u \in \Z_p^\times$, so this is a contradiction. Hence in the presentation $\frac{\Z_p[t]}{(F(t))} \isoto \bT^0$ we must have $\deg(F)=1$ and $\bT^0=\Z_p$.

\end{proof}
\begin{rems} \hfill
\begin{enumerate}
\item In the case of weight $k=2$, the question of when $\bT^0=\Z_p$ was first considered by Mazur \cite[Section II.19, pg.~140]{mazur1978}. In that case, Mazur proved that $I^0$ is always principal \cite[Proposition II.16.6, pg.~126]{mazur1978}. In that case, the analog of our corollary is that $\bT^0=\Z_p$ if and only if $\xi_\mathrm{MT}'$ is a unit, and this was proven by Merel \cite[Th\'{e}or\`{e}me 2]{merel1996}\footnote{In fact, Merel proved that $\bT^0=\Z_p$ if and only if $\displaystyle \prod_{i=1}^{\frac{N-1}{2}}i^i$ is a $p$-th power modulo $N$, which is equivalent to $\xi_\mathrm{MT}'$ being a unit by a non-trivial (but elementary) computation (see \cite[Proposition 1.2]{lecouturier2018}). We learned of this equivalence from Akshay Venkatesh, who discovered it together with Frank Calegari.}. Our proof of the corollary is inspired by Lecouturier's recent new proof of Merel's result \cite[Theorem 1.1]{lecouturier2017arxiv}. In \cite{PG3}, we gave a completely different proof of Merel's result using deformation theory, which is related to the discussion in the Section \ref{subsec:galois reps when I principal} below.
\item One can check computationally that $I^0$ is very often principal, but not always. Indeed, if Remark \ref{rem:R=T} is correct, then $I^0$ should be principal if and only if the equality \eqref{eq:sum of logs} fails to hold. See \cite{deo2021} for some recent results about this.
\item Using the same methods as \cite{PG3}, we could prove directly that, if $J^{\min}$ is principal, then $R_\Db$ is a free $\Z_p$-module of rank $2$ if and only if, in the notation of Proposition \ref{prop:c cup log}, $c \cup \log_N \ne 0$. 
\end{enumerate}
\end{rems}

\begin{cor}
\label{cor:tildef}
Suppose that $\Z_p=\bT^0$, so that there is a unique element $\tilde{f} \in S_k(\Gamma_0(N),\Z_p)_\mathrm{Eis}$ with $a_1(\tilde{f})=1$ and it is an eigenform. Then $\tilde{f}$ is a lift of the form $f$ of Corollary \ref{cor:z mod p^2nu}. In particular, we have 
\[
a_\ell(\tilde{f}) \equiv 1+\ell^{k-1} + p^{\nu+v_p(k)}(1-\ell^{k-1})\log_N(\ell)\alpha \pmod{p^{2\nu+v_p(k)}}
\]
for all $\ell \ne N$, where $\alpha$ is the extra reducibility constant \eqref{eq:alpha}.
\end{cor}
\begin{proof}
Since $\bT^0=\Z_p$, there is a unique $\Z_p$-algebra homomorphism $\bT^0 \to \Z/p^{2\nu+v_p(k)}\Z$ and it is given by $T_\ell \mapsto a_\ell(\tilde{f})$. By Theorem \ref{thm:rank1}, we know that $\xi_\mathrm{MT}'$ is a unit, so Corollary \ref{cor:z mod p^2nu} furnishes an explicit homomorphism $\bT^0 \to \Z/p^{2\nu+v_p(k)}\Z$ given by $f$. The coincidence of these two homomorphisms gives the result.
\end{proof}

\subsection{Consequences for Galois representations when $I^0$ is principal}\label{subsec:galois reps when I principal}
In this section, we construct some Galois representations when $I^0$ is principal. 

\begin{cor}
\label{cor:mod p2 repn}
Assume that $\bT^0=\Z_p$. Then there is a representation $\rho:G_{\Q,Np} \to \GL_2(\Z/p^{2\nu+v_p(k)}\Z)$ with
\[
\rho= \ttmat{\kcyc^{k-1}\chi_\alpha^{-1}}{0}{C}{\chi_\alpha}
\]
where $\chi_\alpha: G_{\Q,Np} \to (\Z/p^{2\nu+v_p(k)}\Z)^\times$ is the extra reducibility character \eqref{eq:chialpha} 
and where $C$ satisfies
\begin{enumerate}
\item $C|_{G_p} =0$,
\item the map $\bar{C}:G_{\Q,Np} \to \F_p$ obtained by reducing $C$ has the property that $\bar{C}\omega^{1-k}:G_{\Q,Np} \to \F_p(1-k)$ is a cocycle with non-zero cohomology class.
\end{enumerate} 
\end{cor}
\begin{rem}
The function $C:G_{\Q,Np} \to \Z/p^{2\nu+v_p(k)}\Z$ is not group cocycle in the usual sense, but we do have $C \in Z^1_{G_{\Q,Np}}(\kcyc^{k-1}\chi_\alpha^{-1}, \chi_\alpha)$ for the $\Ext^1$-cocycle group $Z^1_G(-,-)$ defined in \cite[Section 3.1]{bellaiche2012}, so we will still refer to $C$ as a ``cocycle".
\end{rem}
\begin{proof}
First note that, by Theorem \ref{thm:rank1}, our assumption implies that $\xi_\mathrm{MT}'$ is a unit, so the extra reducibility constant $\alpha$ is well-defined. 

Let $\tilde{f}$ be the form defined in Corollary \ref{cor:tildef}, and let $G_{\Q,Np} \to \GL(V_{\tilde{f}}) \cong \GL_2(\Q_p)$ denote the associated Galois representation. The semi-simplification of the reduction of any stable lattice in $V_{\tilde{f}}$ is $\omega^{k-1} \oplus 1$. By Ribet's Lemma \cite[Proposition 2.1]{ribet1976}, we can choose a lattice $T_{\tilde{f}}$ such that the reduction is a non-split extension of $\omega^{k-1}$ by $1$. Choosing an appropriate basis for $T_{\tilde{f}}$, we obtain a representation $\rho_{\tilde{f}}:G_{\Q,Np} \to \GL_2(\Z_p)$ such that
\[
\rho_{\tilde{f}} \otimes \F_p = \ttmat{\omega^{k-1}}{0}{\omega^{k-1}c}{1},
\]
where $c: G_{\Q,Np} \to \F_p(1-k)$ is a cocycle whose cohomology class is non-zero. Since $\tilde{f}$ is ordinary, we know that $\rho_{\tilde{f}}|_{G_p}$ is upper triangular, so $c|_{G_p}=0$. 

Let $\rho =\rho_{\tilde{f}} \otimes_{\Z_p} \Z/p^{2\nu+v_p(k)}\Z$. By Theorem \ref{thm:T0red}, we know that the pseudorepresentation associated to $\rho$ is the reduction modulo $\xi_\mathrm{MT}^\mathrm{Eis}$ of the universal reducible deformation of Lemma \ref{lem:Rred}. In our current notation, this reduction is $\kcyc^{k-1}\chi_\alpha^{-1} \oplus \chi_\alpha$. Since $\rho$ is reducible as a pseudorepresentation, and $\rho \otimes \F_p$ is lower-triangular but non-split, we see that $\rho$ is lower triangular. This proves that $\rho$ has the desired properties.
\end{proof}

There are also variations in the case $\xi_\mathrm{MT}'$ is not a unit, the simplest being the following.

\begin{cor}
\label{cor:dual numbers rep}
Assume that $I^0$ is principal, that $\xi_\mathrm{MT}'$ is not a unit, and that $\nu=1$ and $v_p(k)=0$. Then there is a representation $\rho:G_{\Q,Np} \to \GL_2(\F_p[X]/(X^2))$ with 
\[
\rho= \ttmat{\kcyc^{k-1}\chi^{-1}}{0}{C}{\chi}
\]
where $\chi: G_{\Q,Np} \to (\F_p[X]/(X^2))^\times$ denotes
the character $\chi(\sigma)=1+\log_N(\sigma)X$, and where $C$ satisfies the same conditions as the previous corollary.
\end{cor}
\begin{proof}
The assumptions imply that $\bT^0$ is a ramified DVR, so there is a unique eigenform and it has coefficients in the fraction field of $\bT^0$. The representation $\rho$ is obtained by taking a $\bT^0$-lattice in its Galois representation, and reducing modulo $(I^0)^2+p\bT^0$. The properties are proven just as in the last corollary.
\end{proof}

\begin{rem}
\label{rem:why lifting}
Note that in the construction of these Galois representations, the only reason we need the assumption that $I^0$ is principal is in order to lift to characteristic zero and apply Ribet's Lemma. It likely that these representations can be constructed directly using the geometry of modular curves, without the assumption that $I^0$ is principal. Morally, the representations should exist simply because $f$ is a \emph{cuspidal} eigenform. This raises the question: what does the condition of being ``cuspidal" mean in the deformation ring?
\end{rem}

Now consider the element\footnote{The fact that $\xi_\mathrm{MT}^\star$ is in $\Lambda$ is not automatic, but follows easily from the fact that $\xi_\mathrm{MT}$ is in $\Lambda$.} $\xi_\mathrm{MT}^\star\in \Lambda$ defined as
\begin{equation}
\label{eq:MT star}
\xi_\mathrm{MT}^\star(\chi) := \begin{cases}
\zeta(1-k)(1-N^k) & \chi=\mathbbm{1} \\
L(1-k,\chi) & \chi \ne \mathbbm{1}.
\end{cases}
\end{equation}
Note that $\xi_\mathrm{MT}^\star \ne \xi_\mathrm{MT}$, but we think of $\xi_\mathrm{MT}^\star$ as the alteration of $\xi_\mathrm{MT}$ with dual local condition at $N$. The image of $\xi_\mathrm{MT}^\star$ in $\Lambda_1$ is
\[
\xi_\mathrm{MT}^\star \equiv \zeta(1-k)(1-N^k) + \xi_\mathrm{MT}'X.
\]

In the next theorem, we use the notation $H^1_{(p)}(\Z[1/Np],-)$ for the trivial-at-$p$ Selmer group -- see Appendix \ref{app:galois} for the definition.
\begin{thm}
\label{thm:lifting c}
Assume that $I^0$ is principal. If $\xi_\mathrm{MT}'$ is a unit, then map
\[
H^1_{(p)}(\Z[1/Np],(\Lambda/\xi_\mathrm{MT}^\star)(1-k)) \to H^1_{(p)}(\Z[1/Np],\F_p(1-k)),
\]
induced by the quotient map in the coefficients, is non-zero. 

If $ \xi_\mathrm{MT}'$ is not a unit,  $\nu=1$, and $v_p(k)=0$, then the image of $\xi_\mathrm{MT}^\star$ in $\bar\Lambda_1$ is zero, and the map
\[
H^1_{(p)}(\Z[1/Np],\bar\Lambda_1(1-k)) \to H^1_{(p)}(\Z[1/Np],\F_p(1-k)),
\]
induced by the quotient map in the coefficients, is non-zero. 
\end{thm}
\begin{proof}
When $\xi_\mathrm{MT}'$ is a unit, we have
\[
\Lambda/\xi_\mathrm{MT}^\star \cong \Z/p^{2\nu+v_p(k)}\Z
\]
by $X \mapsto -\frac{\xi_\mathrm{MT}^\star(\mathbbm{1})}{\xi_\mathrm{MT}'}$. This shows that $G_{\Q,Np}$ acts $(\Lambda/\xi_\mathrm{MT}^\star)(1-k)$ as the character $\kcyc^{1-k}\chi_{\alpha}^{-2}$. By Corollary \ref{cor:mod p2 repn} we have the representation $\rho$, and we see that the extension class $C \chi_\alpha^{-1}$ defined by $\rho \otimes \chi_\alpha^{-1}$ is in $H^1_{(p)}(\Z[1/Np],(\Lambda/\xi_\mathrm{MT}^\star)(1-k))$ and has non-zero reduction. 

In the case $\xi_\mathrm{MT}'$ is not a unit,  $\nu=1$, and $v_p(k)=0$, we have the representation $\rho$ of Corollary \ref{cor:dual numbers rep}.
\end{proof}

\subsection{Algebraic number theory consequence} Our Theorem \ref{thm:lifting c} has the following consequence.
\begin{thm}
\label{thm:IMC modular}
Assume that $\bT^0=\Z_p$. We have $\Ann_\Lambda(H^2_{(p)}(\Z[1/Np],\Lambda(1-k)))= \xi_\mathrm{MT}^\star\Lambda$.
\end{thm}
The proof that Theorem \ref{thm:IMC modular} follows from Theorem \ref{thm:lifting c} is given in Proposition \ref{prop:lifting and IMC} below. In Section \ref{sec:lift cup slope}, we also give other interpretations of this theorem in terms of cup products and slopes.

As we will see in the remainder of the paper, this theorem is predicted by an equivariant version of the Bloch-Kato conjecture, as formulated by Kato. Moreover, we show,  using a version of the Equivariant Iwasawa Main Conjecture that has been proven \cite{CS1974,GP2015}, that the equality $\Ann_\Lambda(H^2_{(p)}(\Z[1/Np],\Lambda(1-k)))= \xi_\mathrm{MT}^\star\Lambda$ holds without any assumption (see Theorem \ref{prop:star and non-star IMC} and Theorem \ref{thm:EIMC}).

\part{Tame Bloch-Kato conjecture}
\label{Part2}
The purpose of Part 2 is to explain how Theorem \ref{thm:IMC modular} fits into to the general framework of special values conjectures. We especially want to address why the ``altered" Mazur-Tate $\zeta$-function $\xi_\mathrm{MT}^\star$ appears (as opposed to the unaltered variant $\xi_\mathrm{MT}$).  We will show that it has to do with the ``trivial at $p$" Selmer condition on the Galois cohomology in Theorem \ref{thm:IMC modular} (as opposed to unaltered Galois cohomology). 

The main new result of Part 2 is Theorem \ref{prop:star and non-star IMC}, where we prove that the annihilator equality in Theorem \ref{thm:IMC modular} is equivalent to an ``unaltered" variant. This unaltered variant is a Coates-Sinnot formulation of the Equivariant Iwasawa Main Conjecture (EIMC), which they also have proven in this case \cite{CS1974}.  By combining Theorem \ref{thm:IMC modular} with Theorem \ref{prop:star and non-star IMC},  we have a new proof of EIMC when $\bT^0 =\Z_p$.

We were unable to prove Theorem \ref{prop:star and non-star IMC} -- the equivalence of the altered and unaltered versions --
using standard techniques of Iwasawa theory (like Fitting ideals, etc.). Instead, following a suggestion of Venkatesh, we attempted to show why Theorem \ref{thm:IMC modular} follows from Kato's formulation of the Bloch-Kato conjecture for families of motives \cite{kato1993a}. It was only in this process  that we saw why both Theorem \ref{prop:star and non-star IMC} and the EIMC follow from Kato's conjecture, and this is the basis of our proof of Theorem \ref{prop:star and non-star IMC}.

In Section \ref{sec:kato conj}, we discuss Kato's conjecture in special case where the family of motives is given by twisting a fixed motive by a tame character. In Section \ref{sec: Q(1-k)}, we further specialize to the case where the motive is $\Q(1-k)$, and show that Kato's conjecture in this case implies Theorem \ref{thm:IMC modular}. By altering the Selmer conditions, we prove Theorem \ref{prop:star and non-star IMC}.  We view Sections \ref{sec:kato conj} and \ref{sec: Q(1-k)} as a kind of ``worked example" of Kato's conjecture; we hope that this has some expository value.  In Section \ref{sec:lift cup slope} (which is independent from the other sections in Part 2), we prove relations between main conjectures formulated in terms of: annihilators of cohomology, cup products,  lifting cohomology classes, and slopes. These results explain why Theorem \ref{thm:lifting c} implies Theorem \ref{thm:IMC modular}.

Throughout Part 2, we use the notions of determinants and regulators introduced in Appendix \ref{app:alg}. We also use the notation for Galois cohomology established in Appendix \ref{app:galois}; here we give a brief summary of this notation (but see Appendix \ref{app:galois} for the actual definitions):
\begin{itemize}
\item $\RG(\Z[1/Np],-)$ is short hand for the continuous cochain complex of $G_{\Q,Np}$,
\item $\RG_{(p)}(\Z[1/Np],-)$ (resp.~$\RG_{(N)}(\Z[1/Np],-)$) is the Selmer complex of $G_{\Q,Np}$-cohomology with the ``trivial" condition at $p$ (resp.~at $N$) and no condition at other places,
\item $\RG_{c}(\Z[1/Np],-)$ is ``compactly supported cohomology" of $G_{\Q,Np}$,
\item $\RG(\Q_\ell,-)$ is local Galois cohomology,
\item $\RG_f(\Q_\ell,-)$ is the Bloch-Kato finite cohomology (if $\ell=p$) or unramified cohomology (if $\ell \ne p$),
\item $\RG_{/f}(\Q_\ell,-)$ is the ``non-finite" cohomology (i.e. ~ the cone of $\RG_f \to \RG$),
\item $\RG_f(\Z[1/Np],-)$ is the Bloch-Kato Selmer complex.
\end{itemize}
We also retain our notation from Section \ref{subsec:setup}, especially the assumptions about the primes $N$ and $p$ and the integer $k$.

\section{Kato's main conjecture for tame families} 
\label{sec:kato conj}
Bloch and Kato \cite{BK1990} formulated a beautiful conjecture explaining the arithmetic content of special values of motivic $L$-functions. Kato \cite{kato1993,kato1993a}\footnote{A similar formulation was found independently by Fontaine and Perrin-Riou \cite{FP1994}.} later formulated a version of the conjecture that allows for the consideration of \emph{families} of motives. Central to Kato's formulation is the idea of \emph{zeta elements}. In this section, we discuss the relevant special case of Kato's conjecture. We frequently refer to the nice survey \cite[Part I]{flach2004}, which contains more detail and considers the general case.

\subsection{Setup} 
We consider a pure motive $M$ that has good reduction at $N$ and $p$, and let $S$ be the set of primes at which $M$ has bad reduction together with $N$, $p$ and infinity. We consider the family of motives $\{M(\chi)\}_{\chi}$ that are twists of $M(\mathbbm{1}):=M$ by Dirichlet characters $\chi$ of conductor dividing $N$ and $p$-power order (for the rest of this section, $\chi$ will always refer to such a character). We assume that the Betti and de Rham realizations of each $M(\chi)$ satisfy
\begin{equation}
\label{Deligneperiod}
 H_B(M(\chi))^+ = 0, (H_{dR}/F^0 H_{dR})(M(\chi)) = 0,
\end{equation}
where the superscript ``$+$" indicates the part fixed by complex conjugation.
 
From now on, we only consider the $p$-adic \'etale realizations $M_p(\chi)$ of the $M(\chi)$. We let $T_M \subset M_p$ denote a stable $\Z_p$-lattice in $M_p$. Note that $\bar\Q_p$-points of $\Lambda$ correspond to Dirichlet characters $\chi$ as above, and that for any such point $\chi:\Lambda \to \bar\Q_p$, we have $M_p(\chi)= (T_M \otimes_{\Z_p} \Lambda) \otimes_\Lambda \bar\Q_p$. In other words, the $p$-adic realizations $M_p(\chi)$ are the points in the family $T_M \otimes_{\Z_p} \Lambda$.

\subsection{Kato's main conjecture}
\label{subsec:kato conj} In this setting, Kato's main conjecture states that there is a canonical integral generator
\begin{equation}
\label{eq:kato conjecture}
s_c \in {\det}_\Lambda(\RG_c(\Z[1/S],T_M \otimes_{\Z_p} \Lambda)),
\end{equation}
called the \emph{zeta element}\footnote{Note that this conjecture is independent of the choice of $T_M$ \cite[Remark 4.10]{kato1993a}.} (see \cite[Conjecture 3 on pg.~6]{flach2004}).  Assuming the Deligne-Beilinson conjecture and regulator conjectures, the zeta element can be described in terms the $L$-values $L(M(\chi),0)$, as we now sketch.

\subsubsection{Sketch of the origin of $s_c$} We sketch the conjectural construction of $s_c$, following \cite{flach2004}. For each character $\chi$, there is a canonical $\Q(\chi)$-vector space, denoted $\Xi(M(\chi))$ in \cite{flach2004}, built out of determinants of $H_B(M(\chi))^+$, $(H_{dR}/F^0 H_{dR})(M(\chi))$, and the motivic cohomology of $M(\chi)$ and its dual. The $p$-adic regulator conjecturally induces a canonical isomorphism 
\[
\vartheta_{p}:\Xi(M(\chi)) \otimes_\Q \Q_p \isoto {\det}_{\Q_p(\chi)}(\RG_c(\Z[1/S],M_p(\chi))).
\] 
On the other hand, Beilinson's regulator gives a conjectural canonical isomorphism 
\[
\vartheta_\infty: \R \isoto \Xi(M(\chi)) \otimes_\Q \R,
\] and the Deligne-Beilinson conjecture \cite{deligne1979,beilinson1984} states that $\vartheta_\infty(L(M(\chi),0)^{-1})$ is in $\Xi(M(\chi))$. Assuming all these conjectures, we have a canonical element
\[
s_c(\chi) := \vartheta_p(\vartheta_\infty(L(M(\chi),0)^{-1})) \in {\det}_{\Q_p(\chi)}(\RG_c(M_p(\chi))).
\]
Then Kato's conjecture is as follows.

\begin{conj}[Kato]
\label{conj:kato}
The sections $s_c(\chi)$ glue to give an integral section 
\[
s_c \in {\det}_\Lambda(\RG_c(\Z[1/S],T_M \otimes_{\Z_p} \Lambda))
\]
and $s_c$ is a generator of this free $\Lambda$-module.
\end{conj}

\begin{rem}
There is an analogy between Galois cohomology and cohomology of three-manifolds (see, for example, \cite{mazur1973}). The existence of a canonical generator \eqref{eq:kato conjecture} can be thought of as an instance of this analogy. Indeed, the cohomology of a manifold can be computed by taking a triangulation, and the resulting element in the determinant of cohomology is independent of the choice of triangulation. Hence the existence of the element $s_c$ can be though of as analogous to the existence of a triangulation on a three-manifold. This analogy was explained to us by Venkatesh.
\end{rem}

\subsubsection{Characterization of $s_c$ in terms of zeta values} The specializations $s_c(\chi) \in {\det}_{\Q_p(\chi)}(\RG_c(M_p(\chi)))$ of $s_c$ are characterized by zeta values. This characterization involves a related section $s_f(\chi)$ of ${\det}_{\Q_p(\chi)}(\RG_f(M_p(\chi)))$ that we now define.

For each $\chi$, we have an isomorphism 
\[
\hspace*{-0.3in}
{\det}_{\Q_p(\chi)}(\RG_f(\Z[1/S],M_p(\chi))) = {\det}_{\Q_p(\chi)}(\RG_c(\Z[1/S],M_p(\chi))) \otimes \left( \bigotimes_{s \in S} {\det}_{\Q_p(\chi)}(\RG_f(\Q_s,M_p(\chi))) \right)
\]
We need to define a section of ${\det}_{\Q_p(\chi)}(\RG_f(\Q_s,M_p(\chi)))$ for each $s$, where $\ell$ denotes a finite place of $S$ other than $p$:
\begin{description}
\item[$s=\infty$] We have $\RG_f(\R,M_p(\chi)) \simeq 0$ by\eqref{Deligneperiod}, so there a canonical element $s_{\R}(\chi) \in {\det}_{\Q_p(\chi)}(\RG_f(\R,M_p(\chi)))$.
\item[$s=\ell$] We have $\RG_f(\Q_\ell,M_p(\chi))=[M_p(\chi)^{I_\ell} \xrightarrow{1-\Fr_\ell} M_p(\chi)^{I_\ell}]$, so there a canonical element $s_{\Q_\ell}(\chi) \in {\det}_{\Q_p(\chi)}(\RG_f(\Q_\ell,M_p(\chi)))$ by Example \ref{eg:endomorphism complex}.
\item[$s=p$] We have $D_\mathrm{dR}(M_p(\chi))/ D^0_\mathrm{dR}(M_p(\chi))=0$ by \eqref{Deligneperiod}, so $\RG_f(\Q_p,M_p(\chi))=[D_\mathrm{crys}(M_p(\chi)) \to D_\mathrm{crys}(M_p(\chi))]$, so there a canonical element $s_{\Q_p}(\chi) \in {\det}_{\Q_p(\chi)}(\RG_f(\Q_p,M_p(\chi)))$ by Example \ref{eg:endomorphism complex}.
\end{description}
We can then define a section $s_f(\chi) \in {\det}_{\Q_p(\chi)}(\RG_f(\Z[1/S],M_p(\chi)))$ by
\begin{equation}
\label{eq:def of sf} 
s_f(\chi) = s_c(\chi) \otimes \left( \bigotimes_{s \in S}  s_{\Q_s}(\chi)\right).
\end{equation}
With this setup, Kato's Conjecture \ref{conj:kato} implies that
\begin{equation}
\label{eq:BK} 
\mathrm{reg}( s_f(\chi))^{-1} = L(M(\chi),0).
\end{equation}

Note that the function $\chi \mapsto L(M(\chi),0)$ may not be an element of $\Lambda$. This is a reflection of the fact that, although the sections $s_c(\chi)$ glue to give a section $s_c$ over the whole family, the sections $s_f(\chi)$ may not. Indeed, it may happen that $\RG_f(T_M \otimes \Lambda)$ is defined, but not a perfect complex. In the next section, we will examine a case where this happens because $(T_M \otimes \Lambda)^{I_N}$ is not a perfect $\Lambda$-module, which, in turn, causes $\RG_f(\Q_N,T_M \otimes \Lambda)$ not to be perfect.

\section{The case $M=\Q(1-k)$}
\label{sec: Q(1-k)}
In this section, we specialize the discussion of the previous section to the case $M=\Q(1-k)$ with $k \ge 2$ even. We write $\Q_p(\chi)(1-k)$ for $M_p(\chi)$. In this case, we have $S=\{N,p,\infty\}$ and we write $\RG(\Z[1/Np],-)$ (and similar) instead of $\Z[1/S]$. We may take $T_M=\Z_p(1-k)$, so we have $T_M \otimes \Lambda = \Lambda(1-k)$. In this section, since $N$ is the only finite prime of $S$ that is not $p$, we emphasize its importance by defining $\RG_\ur(\Q_N,\Lambda(1-k)):=\RG_f(\Q_N,\Lambda(1-k))$ and $\RG_{/\ur}(\Q_N,\Lambda(1-k)):= \RG_{/f}(\Q_N,\Lambda(1-k))$, and we denote the section $s_{\Q_N}(\chi) \in \RG_\ur(\Q_N,M_p(\chi))$ defined in the previous section by $s_\ur(\chi)$.

\subsection{Imperfect complexes and the failure of $p$-adic continuity} For $M=\Q(1-k)$, the conjectural formula \eqref{eq:BK} becomes
\begin{equation}
\label{eq:BK for 1-k} 
\mathrm{reg}( s_f(\chi))^{-1} =\begin{cases} \zeta(1-k) & \chi=\mathbbm{1} \\
L(1-k,\chi) & \chi \ne \mathbbm{1}. \end{cases}
\end{equation}
For the remainder of Section \ref{sec: Q(1-k)}, we will refer to equation \eqref{eq:BK for 1-k} (as well as the existence of the element $s_c$, which is used to define $s_f$) as ``Kato's conjecture".

Note that there is no Euler factor for $\chi=\mathbbm{1}$, so $\mathrm{reg}( s_f(\chi))^{-1}$ is \textbf{not} equal to the Mazur-Tate $\zeta$-function $\xi_\mathrm{MT}$. In fact, it's easy to see that the function $\chi \mapsto \mathrm{reg}( s_f(\chi))^{-1}$ is not even in $\Lambda$; we call this the \emph{failure of $p$-adic continuity}.

There is a conceptual reason for this failure of $p$-adic continuity. In addition to \eqref{Deligneperiod}, we also have 
\[
\RG_f(\Q_p,\Q_p(\chi)(1-k)) = [\Q_p(\chi) \xrightarrow{1-p^{k-1}} \Q_p(\chi)] 
\]
for every $\chi$. In this case, it is reasonable to define 
\[
\RG_f(\Q_p,\Lambda(1-k)) = [\Lambda \xrightarrow{1-p^{k-1}} \Lambda]
\]
and $\RG_f(\R,\Lambda(1-k))=0$, so the sections $s_\R(\chi)$ and $s_{\Q_p}(\chi)$ glue to gives sections $s_\R \in {\det}_\Lambda( \RG_f(\R,\Lambda(1-k)))$ and $s_{\Q_p} \in {\det}_\Lambda( \RG_f(\Q_p,\Lambda(1-k)))$. In fact, $s_\R$ is just 1, so we will leave it out below.

We can define $\RG_f(\Z[1/Np],\Lambda(1-k))$ to be the mapping fiber\footnote{By `mapping fiber' we mean $\mathrm{Cone}(-)[-1]$.} of the map
\[
\RG(\Z[1/Np],\Lambda(1-k)) \to \RG_{/f}(\Q_p,\Lambda(1-k)) \oplus \RG_{/\ur}(\Q_N,\Lambda(1-k)),
\]
but note that $\Lambda(1-k)^{I_N}$ is not a perfect $\Lambda$-module, so neither $\RG_{/\ur}(\Q_N,\Lambda(1-k))$ nor $\RG_f(\Z[1/Np],\Lambda(1-k))$ is a perfect complex. This means that ${\det}_\Lambda(\RG_f(\Z[1/Np],\Lambda(1-k)))$ is not even defined, so we cannot hope that the sections $s_f(\chi)$ glue together in a reasonable way. 

\begin{rem}
This failure of continuity is familiar from the study of $p$-adic $L$-functions. In that case, the failure is due to an imperfect local complex at $p$, and the solution is to change an Euler factor at $p$. Here, the failure is due to an imperfect local complex at $N$,  and, as we will see in Section \ref{subsub:p-adic cont}, the solution is to change the Euler factor at $N$.
\end{rem}

 \subsection{Kato's conjecture implies $p$-adic continuity}
 \label{subsub:p-adic cont}
 We will produce a better result by replacing $\RG_f(\Z[1/Np],\Lambda(1-k))$ with $\RG_{(p)}(\Z[1/Np],\Lambda(1-k))$, which continues to impose the finiteness condition at $p$, but has \emph{no condition} at $N$. The cohomology $\RG_{(p)}(\Z[1/Np],\Lambda(1-k))$ is defined as the mapping fiber of
 \[
 \RG(\Z[1/Np],\Lambda(1-k)) \to \RG_{/f}(\Q_p,\Lambda(1-k))
 \]

The advantage of this cohomology is that $\RG(\Q_N,\Lambda(1-k))$ is computed by a perfect complex of $\Lambda$-modules. The following lemma will be proven in Section \ref{subsec:s_N} below.

\begin{lem}
\label{lem:s_N}
There integral generator $s_N$ of ${\det}_\Lambda(\RG(\Q_N,\Lambda(1-k)))$ and, for each $\chi$, a generator $s_{/ \ur}(\chi) \in \det_{\Q_p(\chi)}\RG_{/ \ur}(\Q_N,M_p(\chi))$ such that 
\[
s_N(\chi)=s_{\ur}(\chi) \otimes s_{/ \ur}(\chi)
\]
and
\begin{equation}
\label{eq:reg sur}\reg(s_{\ur}(\chi)) =\left.\begin{cases}
1-N^{k-1} & \chi=1 \\
1 & \chi \ne 1
\end{cases}\right\}, \ 
\reg(s_{ / \ur}(\chi)) =\left.\begin{cases}
(1-N^{k})^{-1} & \chi=1 \\
1 & \chi \ne 1.
\end{cases}\right\}
\end{equation}
\end{lem}

We have the short exact sequence
\[
\RG_{c}(\Lambda(1-k)) \to \RG_{(p)}(\Lambda(1-k)) \to \RG(\Q_N,\Lambda(1-k)) \oplus \RG_f(\Q_p,\Lambda(1-k)),
\]
where we have dropped the ``$\Z[1/Np]$" in the first two terms for brevity.
Using this lemma, we can, assuming the existence of Kato's element $s_c$, define an integral generator $s_{(p)} \in {\det}_\Lambda(\RG_{(p)}(\Z[1/Np],\Lambda(1-k)))$ as
\[
s_{(p)} = s_c \otimes s_{\Q_p} \otimes s_N.
\]
Comparing this to \eqref{eq:def of sf} and using the fact that $s_N(\chi)=s_{\ur}(\chi) \otimes s_{/ \ur}(\chi)$, we see that, for any $\chi$, we have
 \[
 s_{(p)}(\chi)=s_f(\chi) \otimes s_{/ \ur}(\chi).
 \]
The formula \eqref{eq:reg sur} together with Kato's conjecture \eqref{eq:BK for 1-k} implies
\begin{equation}
\label{eq:BK xiMT star version}
 \reg(s_{(p)})^{-1}(\chi) =\begin{cases}
\zeta(1-k)(1-N^{k}) & \chi=1 \\
L(1-k,\chi) & \chi \ne 1.
\end{cases}
\end{equation}
In other words, the conjecture says that $\reg(s_{(p)})^{-1}=\xi_\mathrm{MT}^\star$, where $\xi_\mathrm{MT}^\star$ is the modified Mazur-Tate $L$-function defined in \eqref{eq:MT star}\footnote{Alternatively, we could have worked with cohomology with the vanishing-at-$N$ condition, and this process would yield the usual Mazur-Tate $L$-function $\xi_\mathrm{MT}$.}.

The conjecture \eqref{eq:BK xiMT star version} has to do with the value of the regulator on a special generator that comes from the zeta element. 
If we only care about the \emph{regulator ideal}, as in Definition \ref{defn:regulator},  then this gives something closer to the classical Iwasawa main conjecture, in that it relates the \emph{ideal generated by} the $L$-function to an ideal measuring the \emph{size} of Galois cohomology. 
\begin{lem}
\label{lem:BK implies IMC}
The complex $\RG_{(p)}(\Z[1/Np],\Lambda(1-k))$ is quasi-isomorphic to a complex $
[0 \to \Lambda \xrightarrow{\xi_\mathrm{alg}^\star} \Lambda]$
for some non-zero-divisor $\xi_\mathrm{alg}^\star \in \Lambda$. If Kato's conjecture \eqref{eq:BK xiMT star version} is true, then $\xi_\mathrm{alg}^\star\Lambda = \xi_\mathrm{MT}^\star\Lambda$.
\end{lem}
\begin{proof}
The fact that $\RG_{(p)}(\Z[1/Np],\Lambda(1-k))$ is a perfect complex of $\Lambda$-modules follows from finiteness results in Galois cohomology, as in \cite[Proposition 4.17]{kato1993a}. Simple computations shows that $\dim_{\F_p}H^i_{(p)}(\Z[1/Np],\F_p(1-k))=1$ for $i=1,2$, and zero otherwise, and that $\RG_{(p)}(\Z[1/Np],\Lambda\left[\frac{1}{p}\right](1-k))$ is acyclic. Hence the hypotheses of Lemma \ref{lem:simple complex} are satisfied, and this yields the desired quasi-isomorphism.

We see that $\reg(e_2/e_1)^{-1}=\xi_\mathrm{alg}^\star$, so regulator ideal $\reg(\RG_{(p)}(\Z[1/Np],\Lambda(1-k)))$ is $(\xi_\mathrm{alg}^\star)^{-1}\Lambda$. On the other hand, if \eqref{eq:BK xiMT star version} is true, this implies 
\[\reg(\RG_{(p)}(\Z[1/Np],\Lambda(1-k)))=( \xi_\mathrm{MT}^\star)^{-1}\Lambda,
\] so $\xi_\mathrm{alg}^\star\Lambda = \xi_\mathrm{MT}^\star\Lambda$.
\end{proof}

The lemma implies that $\xi_\mathrm{alg}^{\star}\Lambda=\Ann_\Lambda(H^2_{(p)}(\Z[1/Np],\Lambda(1-k)))$, so Kato's conjecture \eqref{eq:BK xiMT star version} implies that $\Ann_\Lambda(H^2_{(p)}(\Z[1/Np],\Lambda(1-k)))= \xi_\mathrm{MT}^\star\Lambda$, which is what we proved in Theorem \ref{thm:IMC modular} under some additional assumptions. In fact, as we see in the next section, the equality $\Ann_\Lambda(H^2_{(p)}(\Z[1/Np],\Lambda(1-k)))= \xi_\mathrm{MT}^\star\Lambda$ can be proven directly using the work of Coates--Sinnot \cite{CS1974}.

We do not know how to construct the special generator $s_{(p)}$ or approach Kato's conjecture \eqref{eq:BK xiMT star version} using the modular methods of Part 1.  However, we see from Lemma \ref{lem:BK implies IMC} that our Theorem \ref{thm:IMC modular} is predicted by conjecture \eqref{eq:BK xiMT star version}. So we can think of Theorem \ref{thm:IMC modular} as evidence for conjecture \eqref{eq:BK xiMT star version} coming from modular forms.

\subsection{Comparison with the equivariant main conjecture}
The equality of ideals $\Ann_\Lambda(H^2_{(p)}(\Z[1/Np],\Lambda(1-k)))= \xi_\mathrm{MT}^\star\Lambda$ is equivalent to a known form of the Equivariant Iwasawa Main Conjecture:
\begin{thm}
\label{prop:star and non-star IMC}
The following equalities of ideals in $\Lambda$ are equivalent:
\begin{enumerate}
\item $\Ann_\Lambda(H^2_{(p)}(\Z[1/Np],\Lambda(1-k)))=\xi_\mathrm{MT}^\star\Lambda$
\item $\Ann_\Lambda(H^2(\Z[1/Np],\Lambda(k)))=\xi_\mathrm{MT}\Lambda$.
\end{enumerate}
\end{thm}
\begin{proof}
Let $\xi_\mathrm{alg}^\star \in \Lambda$ be as in Lemma \ref{lem:BK implies IMC}, so 
\[\xi_\mathrm{alg}^{\star}\Lambda=\Ann_\Lambda(H^2_{(p)}(\Z[1/Np],\Lambda(1-k))).
\]
By Poitou-Tate duality\footnote{See \cite[Theorem (6.3.4)]{nekovar2006} for the version we need (for complexes with Selmer conditions). For a more down-to-earth treatment in terms of cohomology, see \cite[Appendix B]{GV2018}; see \cite{sharifiFKnote} or the proof of \cite[Proposition (5.2.4)]{nekovar2006} for the technique used to upgrade from cohomology to complexes.} we have an isomorphism
\[
R\Hom_\Lambda(\RG_{(p)}(\Z[1/Np],\Lambda(1-k)),\Lambda)[-3] \cong \RG_{(N)}(\Z[1/Np],\Lambda(k)).
\]
By Lemma \ref{lem:BK implies IMC}, this implies that there is a quasi-isomorphism
\[
\RG_{(N)}(\Z[1/Np],\Lambda(k)) \simeq [0 \to \Lambda \xrightarrow{\xi_\mathrm{alg}^\star} \Lambda]
\]
and we see that $\reg(\RG_{(N)}(\Z[1/Np],\Lambda(k)))^{-1}= {\xi_\mathrm{alg}^\star}\Lambda$.

Just as in the proof of Lemma \ref{lem:BK implies IMC}, a simple computation verifies the hypotheses of Lemma \ref{lem:simple complex} for $\RG(\Z[1/Np],\Lambda(k))$, so there is a quasi-isomorphism
\[
\RG(\Z[1/Np],\Lambda(k)) \simeq [0 \to \Lambda \xrightarrow{\xi_\mathrm{alg}} \Lambda].
\]
for some non-zero divisor $\xi_\mathrm{alg} \in \Lambda$, and we have 
\[
\reg(\RG(\Z[1/Np],\Lambda(k)))^{-1}= {\xi_\mathrm{alg}}\Lambda=\Ann_\Lambda(H^2(\Z[1/Np],\Lambda(k))).
\]

Now, considering the triangle
\[
\RG_{(N)}(\Z[1/Np],\Lambda(k)) \to \RG(\Z[1/Np],\Lambda(k)) \to \RG(\Q_N,\Lambda(k)),
\]
we see that, $\xi_\mathrm{alg} =u \xi_\mathrm{alg}^\star \reg(s_N)$ for some $u \in \Lambda^\times$. Considering the formula \eqref{eq:reg sur} for $\reg(s_N)$, we also have $\xi_\mathrm{MT}=\xi_\mathrm{MT}^\star \reg(s_N)$, and this completes the proof.
\end{proof}
The equality $\Ann_\Lambda(H^2(\Z[1/Np],\Lambda(k)))=\xi_\mathrm{MT}\Lambda$ is a (very) special case of the Coates-Sinnot conjecture \cite{CS1974}.
\begin{thm}[Coates--Sinnot] \label{thm:EIMC} 
We have $\Ann_\Lambda(H^2(\Z[1/Np],\Lambda(k)))=\xi_\mathrm{MT}\Lambda$.
\end{thm}

This conjecture has multiple known proofs. The results of \cite{CS1974} show that $\xi_\mathrm{MT}$ is in the annihilator, and the result follows from this by a simple argument. It is also proven by Greither and Popescu \cite[Theorem 6.12]{GP2015} as a consequence of their proof of their Equivariant Iwasawa Main Conjecture, which, in turn, they show to follow from the Iwasawa Main Conjecture for totally real fields,  due to Wiles \cite{wiles1990}, and the vanishing of $\mu$-invariants for abelian number fields, due to Ferrero-Washington \cite{FW1979}.

\subsection{Computation of the local generator $s_N$}
\label{subsec:s_N}
In this section, we give the proof of Lemma \ref{lem:s_N}.

By Lemma \ref{lem:local complex}, $H^*(\Q_N,\Lambda(1-k))$ is computed by the total complex of the bicomplex
\begin{equation}
\label{eq:local complex}
\xymatrix{
\Lambda(1-k)e_0 \ar[r]^{X} \ar[d]^{1-\mathrm{Fr}_N^{-1}} & \Lambda(1-k)e_1 \ar[d]^{1-\mathrm{Fr}_N^{-1}\mathcal{N}} \\
\Lambda(1-k)f_1 \ar[r]^{X}  & \Lambda(1-k)e_2
}
\end{equation}
where the horizontal arrows are multiplication by the element $X$ of $\Lambda$ defined in Section \ref{subsec:def of Lambda},  and $\mathcal{N}$ denotes multiplication by the element $\sum_{i=0}^{N-1} (1+X)^i$ of $\Lambda$. We define $s_N=\frac{e_0e_2}{e_1 \wedge f_1}$. The computation of $\reg(s_N)$ can be done directly using Example \ref{eg:regular of a square}, but we compute on each specialization because we think this makes the computation clearer.

For $\chi=\mathbbm{1}$, $X$ maps to $0$ and $\mathcal{N}$ maps to multiplication-by-$N$. Choosing an isomorphism $\Q_p(1-k) \cong \Q_p$ of $\Q_p$-vector spaces, the complex \eqref{eq:local complex} can be identified with
\[
\xymatrix{
\Q_pe_0 \ar[r]^{0} \ar[d]^{1-N^{k-1}} & \Q_pe_1 \ar[d]^{1-N^k} \\
\Q_pf_1 \ar[r]^{0}  & \Q_pe_2.
}
\] 
Under this identification, $\RG_\ur(\Q_N,\Q_p(1-k))$ is the subcomplex 
\[\Q_p(1-k)e_0 \xrightarrow{1-N^{k-1}} \Q_p(1-k)f_1\]
 and $s_{\ur}(\mathbbm{1})=e_0/f_1$, so we see that $\reg(s_{\ur}(\mathbbm{1})) = 1-N^{k-1}$. We can define $s_{/ \ur}(\mathbbm{1}) = s_N(\mathbbm{1}) \otimes s_{\ur}(\mathbbm{1})^{-1}$ and we see $\reg(s_{/ \ur}(\mathbbm{1})) =(1-N^{k})^{-1}$.

For $\chi \ne \mathbbm{1}$, $X$ maps to a non-zero element $x \in \Q_p(\chi)$ and $\mathcal{N}$ maps to the identity.  Choosing an isomorphism $\Q_p(\chi)(1-k) \cong \Q_p(\chi)$ of $\Q_p$-vector spaces, the complex \eqref{eq:local complex} becomes
\[
\xymatrix{
\Q_p(\chi)e_0 \ar[r]^{x} \ar[d]^{1-N^{k-1}} & \Q_p(\chi)e_1 \ar[d]^{1-N^{k-1}} \\
\Q_p(\chi)f_1 \ar[r]^{x}  & \Q_p(\chi)e_2.
}
\]
The complex $\RG_\ur(\Q_N,M_p(\chi))$ is acyclic and $\reg(s_{\ur}(\chi))=1$, so we define $s_{/ \ur}(\chi)=s_N(\chi)$, and we see that $\reg(s_{/ \ur}(\chi))=\reg(s_N(\chi))=1$.

 \section{Interpretation in terms of lifting, cup products, and slopes}
 \label{sec:lift cup slope}
In this section, we fix a generator $c \in H^1_{(p)}(\Z[1/Np],(\Z/p^{\nu+v_p(k)}\Z)(1-k))$. We interpret the main conjecture $\Ann_\Lambda(H^2_{(p)}(\Z[1/Np],\Lambda(1-k)))=\xi_\mathrm{MT}^\star\Lambda$ in terms of $c$ in three closely related ways\footnote{We consider (2) only in the case that $\nu=1$ and $v_p(k)=0$, and (3) only when, additionally, $v_p(k-1)=0$.}:
\begin{enumerate}
\item what quotients of $\Lambda(1-k)$ does $c$ lift to?
\item for which cohomology classes $a$ is does the cup product $a \cup c$ vanish?
\item what is the image $c|_N$ of $c$ in $H^1(\Q_N,(\Z/p^{\nu+v_p(k)}\Z)(1-k))$?
\end{enumerate}
In Part 1, we answered (1) in Theorem \ref{thm:lifting c} using modular forms. The main result of this section is Proposition \ref{prop:lifting and IMC}, which allows us to deduce our result about the main conjecture, Theorem \ref{thm:IMC modular}, from Theorem \ref{thm:lifting c}.

The remainder of the section is meant to expose relationships and analogies between this work and others. Results about (2) were obtained in \cite{PG3} and \cite{SS2019}; Section \ref{subsec:cup} explains how this relates to (1). 
Section \ref{subsec:cup and slope} shows that (3) is a tame analog of the algebraic $\mathcal{L}$-invariant that appears in the Gross-Stark conjecture \cite{gross1981,DDP2011}. 

\subsection{The main conjecture and lifting} We first explain how to interpret the main conjecture in terms of lifting the class $c$.

\begin{prop}
\label{prop:lifting and IMC}
The following two conditions are equivalent:
\begin{enumerate}
\item  $ \Ann_\Lambda(H^2_{(p)}(\Z[1/Np],\Lambda(1-k)))\subset \xi_\mathrm{MT}^\star \Lambda$
\item There is a class of $H^1_{(p)}(\Z[1/Np],(\Lambda/\xi_\mathrm{MT}^\star)(1-k))$ that maps to $c$ under the reduction map.
\end{enumerate}
Moreover, if $\xi_\mathrm{MT}'$ is a unit, then (1) implies $ \Ann_\Lambda(H^2_{(p)}(\Lambda(1-k))) = \xi_\mathrm{MT}^\star \Lambda$.
\end{prop}
\begin{proof}
Lemma \ref{lem:lifting} proves the equivalence of (1) and (2). In general, we know that $\Ann_\Lambda(H^2_{(p)}(\Lambda(1-k)))=\xi_\mathrm{alg}^\star\Lambda$ for some non-zero-divisor $\xi_\mathrm{alg}^\star$, so if $\xi_\mathrm{MT}'$ is a unit, the inclusion in (1) must be equality.
\end{proof}

\subsection{Lifting and cup products}
\label{subsec:cup} For the remainder of the section, we assume that $\nu=1$ and $v_p(k)=0$.  We make this assumption so that there is a canonical section of the map
\[
(\Z/p^{2\nu+v_p(k)}\Z)^\times \onto (\Z/p^{\nu+v_p(k)}\Z)^\times.
\] 

We only consider cup products in usual global cohomology $H^i(\Z[1/Np],-)$; when we write $c \cup -$, we are considering the image of $c$ in $H^1(\Z[1/Np],\F_p(1-k))$. We let $\log_p:(\Z/p^2\Z)^\times \to \F_p$ denote the composition $(\Z/p^2\Z)^\times \to (1+p(\Z/p^2\Z))^\times \to \F_p$, where the first map is the projection $x \mapsto \omega^{-1}(x)x$ and the second map is $1+px \mapsto x \pmod{p}$.

\begin{prop}
\label{prop: c cup chi alpha}
Assume that $\nu=1$ and $v_p(k)=0$ and that $\xi_\mathrm{MT}'$ is a unit. Let $\alpha \in \Z/p^2\Z$ and $\chi_\alpha:G_{\Q,Np} \to (\Z/p^2\Z)^\times$ be the extra reducibility constant and character, as in Definition \ref{defn:alpha}. Consider the following conditions:
\begin{enumerate}
\item The map
\[
H^1_{(p)}(\Z[1/Np],(\Lambda/\xi_\mathrm{MT}^\star)(1-k)) \to H^1_{(p)}(\Z[1/Np],\F_p(1-k))
\]
is non-zero,
\item $c \cup \log_p(\chi_\alpha^{-2}\kcyc^{1-k})=0$ in $H^2(\Z[1/Np],\F_p(1-k))$. 
\end{enumerate}
Then (1) implies (2).
\end{prop}
\begin{proof}
Note that, since $\xi_\mathrm{MT}'$ is a unit, $\Lambda/\xi_\mathrm{MT}^\star$ is isomorphic to $\Z/p^2\Z$ with $G_{\Q,Np}$ acting through $\chi_\alpha^{-2}$, just as in the proof of Theorem \ref{thm:lifting c}.

Considering the long exact sequence in cohomology coming from the short exact sequence
\[
\tag{$\star$}
0 \to \F_p(1-k) \to (\Lambda/\xi_\mathrm{MT}^\star)(1-k) \to \F_p(1-k) \to 0,
\]
we see that (1) implies that $c$ is in the kernel of the boundary map
\[
H^1(\Z[1/Np],\F_p(1-k)) \xrightarrow{\partial} H^2(\Z[1/Np],\F_p(1-k))
\]
for $(\star)$. By definition $\partial c$ is given by $d \tilde{c}$, where $\tilde{c}: G_{Q,Np} \to (\Lambda/\xi_\mathrm{MT}^\star)(1-k)$ is a cochain lifting $c$. Taking $\tilde{c}$ to be the lift defined by the canonical splitting $(\Z/p\Z)^\times \to (\Z/p^2\Z)^\times$, we compute easily that $d \tilde{c} =  c \cup \log_p(\chi_\alpha^{-2}\kcyc^{1-k})$, which completes the proof.
\end{proof}
The proof of the next proposition is similar, replacing the sequence $(\star)$ by
\[
0 \to \F_p(1-k) \to \bar \Lambda_1(1-k) \to \F_p(1-k) \to 0.
\]
In this case,  the boundary map is cup product with $\log_N$. 
\begin{prop}
\label{prop:c cup log}
Consider the two conditions:
\begin{enumerate}
\item The map
\[
H^1_{(p)}(\Z[1/Np],\bar\Lambda_1(1-k)) \to H^1_{(p)}(\Z[1/Np],\F_p(1-k))
\]
is non-zero
\item $c\cup \log_N=0$ in $H^2(\Z[1/Np],\F_p(1-k))$.
\end{enumerate}
Then (1) implies (2).
\end{prop}

\subsection{Cup product and slope}
\label{subsec:cup and slope}
For this section, we continue to assume that $\nu=1$ and $v_p(k)=0$, and also assume $v_p(1-k)=0$. We consider the group $H^1(\Q_N,\F_p(1-k))$. Choosing a primitive $p$-th root of unity in $\zeta_{p}\in \Q_N$, we can identify $H^1(\Q_N,\F_p(1-k))$ with $H^1(\Q_N,\F_p)$. We know that $H^1(\Q_N,\F_p)$ has dimension two and that the cup product pairing
\[
H^1(\Q_N,\F_p) \times H^1(\Q_N,\F_p) \xrightarrow{\cup} H^2(\Q_N,\F_p) \cong \F_p
\]
is a symplectic form. Moreover $H^1(\Q_N,\F_p)$ has a basis $\{\lambda,\log_N\}$ where $\lambda: G_{\Q_N} \to \F_p$ is the unramified character sending $\Fr_N$ to $1$. This choice of basis induces an isomorphism
\[
\mathrm{Slope}:\bP(H^1(\Q_N,\F_p)) \isoto \bP^1(\F_p)
\]
which we call the \emph{slope} map. Explicitly, if $L \subset H^1(\Q_N,\F_p)$ is a line, choose a generator $a \in L$ and write
\[
a=x\lambda + y \log_N,
\]
then $\mathrm{Slope}(L):=[x:y]$. If $a$ generates $L$, we define $\mathrm{Slope}(a):=\mathrm{Slope}(L)$. Note that for non-zero $a$ and $a'$ we have
\begin{equation}
\label{eq:cup and slope}
a \cup a' =0 \Longleftrightarrow \mathrm{Slope}(a) = \mathrm{Slope}(a').
\end{equation}
Indeed, since the cup product pairing is symplectic, the equality $a \cup a'=0$ is equivalent to $a$ and $a'$ spanning the same line.

\begin{prop}Assume that $\nu=1$ and $v_p(k)=v_p(1-k)=0$.
\begin{enumerate}
\item We have $c \cup \log_N$ in $H^2(\Z[1/Np],\F_p(1-k))$ if and only if 
\[
\mathrm{Slope}(c|_N)=[0:1].
\]
\item Assume $\xi_\mathrm{MT}'$ is a unit and let $\chi_\alpha$ be the extra reducibility character of Definition \ref{defn:alpha}.
We have $c \cup \log_p(\chi_\alpha^{-2}\kcyc^{1-k})=0$ in $H^2(\Z[1/Np],\F_p(1-k))$ if and only if 
 \[
 \mathrm{Slope}(c|_N) = [(1-k)\cdot \xi_\mathrm{MT}': k\cdot \zeta(1-k)].
 \]
\end{enumerate}
\end{prop}
\begin{rem}
\label{rem:L-invariant}
We think of $\mathrm{Slope}(c|_N)$ as a tame analog of the algebraic $\mathcal{L}$-invariant that appears in the Gross-Stark conjecture. Indeed, 
in \cite[Section 1]{DDP2011},  that $\mathcal{L}$-invariant is expressed as a slope of global $p$-adic cohomology class in terms of local-at-$p$ cohomology. Our $\mathrm{Slope}(c|_N)$ is the slope of a global mod-$p$ cohomology class in terms of local-at-$N$ cohomology -- the adjective ``tame'' refers to fact that $N \ne p$ here.
\end{rem}
\begin{proof}
Consider the commutative diagram
\[\xymatrix{
H^1(\F_p(1-k)) \times H^1(\F_p) \ar[d] \ar[r]^-\cup & H^2(\F_p(1-k)) \ar[d] \\
H^1(\Q_N,\F_p(1-k)) \times H^1(\Q_N,\F_p)  \ar[r]^-\cup & H^2(\Q_N,\F_p(1-k))
}\]
where the vertical arrows are restriction. One can show that the right vertical arrow is an isomorphism, just as in \cite[Lemma 12.1.1]{PG3}. Hence the cup products $c \cup \log_N$ and $c \cup \log_p(\chi_\alpha^{-2}\kcyc^{1-k})$ vanish if and only if their restrictions at $N$ vanish. 

Using the equivalence \eqref{eq:cup and slope}, 
the only thing that remains to show is that 
\[
\mathrm{Slope}(\log_p(\chi_\alpha^{-2}\kcyc^{1-k})|_N)=[(1-k)\cdot \xi_\mathrm{MT}': k\cdot \zeta(1-k)].
\]
We have $\log_p(\chi_\alpha^{-2}\kcyc^{1-k})|_N = \log_p(\chi_\alpha^{-2})|_N + (1-k)\log_p(\kcyc)|_N$.  From the definition of $\alpha$ and $\chi_\alpha$ (Definition \ref{defn:alpha}) we see that for any $\sigma \in G_{\Q,Np}$, we have
\[
\log_p(\chi_\alpha^{-2})(\sigma)=\left(k \cdot \frac{N-1}{p} \cdot \frac{\zeta(1-k)}{\xi_\mathrm{MT}'} \right)\log_N(\sigma).
\]
On the other hand $\log_p(\kcyc)|_N$ is an unramified character sending $\Fr_N$ to $\frac{N-1}{p}$, so $\log_p(\kcyc)|_N=\frac{N-1}{p}\lambda$. Putting these together, we get
\[
\mathrm{Slope}(\log_p(\chi_\alpha^{-2}\kcyc^{1-k})|_N)=\left[(1-k)\frac{N-1}{p}: k \cdot \frac{N-1}{p} \cdot \frac{\zeta(1-k)}{\xi_\mathrm{MT}'} \right]
\]
which is equal to $[(1-k)\cdot \xi_\mathrm{MT}': k\cdot \zeta(1-k)]$.
\end{proof}

\appendix
\part{Appendices}
\section{Algebraic preliminaries}

In this section, we recall some algebra used in Kato's formulation of the main conjecture.
\label{app:alg}
\subsection{Determinant of a perfect complex} We review the theory of determinants, as discussed in \cite[Section 2.1]{kato1993}. In this subsection, $A$ is a commutative ring. A \emph{perfect complex} of $A$-modules is an object $E$ in the derived category of $A$-modules that is represented by a bounded complex of finitely generated projective $A$-modules. We say that a complex of $A$-modules (or a single $A$-module, considered as a complex in degree zero) is \emph{perfect} if its class is perfect.

The determinant functor $\det_A$ is a functor from the category of perfect complexes (with isomorphisms) to the category of invertible $A$-modules (with isomorphisms) with the following properties:
\begin{itemize}
\item The functor ${\det}_A$ is multiplicative in short exact sequences of complexes.
\item For a single finitely generated projective $A$-module $P$ (concentrated in degree 0), then ${\det}_A(P)$ is the highest exterior power of $P$.  (In particular, ${\det}_A(0)=A$.)
\item  If $E=[\cdots P_i \to P_{i+1} \to \cdots]$ with $P_i$ finitely generated projective, then there is a canonical isomorphism ${\det}_A(E) \cong \otimes_i ({\det}_A(P_i))^{(-1)^i}$.
\item If the cohomology modules $H^i(E)$ are all perfect, then there is a canonical isomorphism ${\det}_A(E)  \cong \otimes_i ({\det}_A(H^i(E)))^{(-1)^i}$.
\end{itemize} 

Note that if $A$ is a semi-local ring (and in this paper we only consider ${\det}_A$ for $A=\Lambda$, $A=\Lambda \otimes \Q$, or $A$ a field), then $\det_A(E)$ is a free $A$-module of rank $1$ for any perfect complex $E$. The purpose of considering determinants is to compare different generators of this free module. For us, one source of such generators comes from Kato's Conjecture \ref{conj:kato}. Other, more prosaic, generators come from the following examples.

\begin{eg}[Acyclic complex]\label{eg:acyclic complex}
\hfill
\begin{enumerate}
\item Let $C^\bullet = [M \xrightarrow{\phi} M']$ where $\phi$ is an isomorphism and $M$ and $M'$ are rank-1 free $A$-modules. Then there is a generator of $\det(C^\bullet)$ given by taking $m$ to be any generator of $M$ and taking $\phi(m)$ as generator of $M'$. The resulting generator $\frac{m}{\phi(m)}$ of $\det(C^\bullet)$ is independent of the choice of $m$.
\item Let $C^\bullet = [M \xrightarrow{\phi} M']$ where $\phi$ is an isomorphism and $M$ and $M'$ are free $A$-modules. Then there is a generator of $\det(C^\bullet)$ given by taking $B$ to be any  basis of $M$ and taking $\phi(B)$ as basis of $M'$. The resulting generator $\frac{\wedge B}{\wedge \phi(B)}$ of $\det(C^\bullet)$ is independent of the choice of $B$.
\item More generally, if $C^\bullet = [M_0 \xrightarrow{\delta_0} M_1 \xrightarrow{\delta_1} \dots  \xrightarrow{\delta_{r-1}} M_r]$ is an acyclic complex of free $A$-modules, we can define a generator of $\det(C^\bullet)$ by taking a basis $B_0$ of $M_0$, completing $\delta(B_0)$ to a basis $B_1 \cup \delta(B_0)$ of $M_1$, completing $\delta(B_1)$ to a basis of $M_2$, and so on. The resulting basis of $\det(C^\bullet)$ is independent of the choices.
\end{enumerate}
\end{eg}

\begin{eg}[Endomorphisms]\label{eg:endomorphism complex}
\hfill
\begin{enumerate}
\item Let $C^\bullet = [M \xrightarrow{\phi} M]$ where $M$ is a free $A$-module of rank $1$. Let $m$ be a basis of $M$. The resulting generator $\frac{m}{m}$ of $\det(C^\bullet)$ is independent of the choice of $m$.
\item Let $D^\bullet$ be a perfect complex of free $A$-modules, and let $\phi:D^\bullet \to D^\bullet$ be an endomorphism. Think of $\phi: D^\bullet \to D^\bullet$ as a double complex, and let $C^\bullet$ be the total complex of it. Then $\det(C^\bullet) \cong \det([\det(D^\bullet) \to \det(D^\bullet)])$, so, by the previous example, we have a canonical generator of $\det(C^\bullet)$.
\end{enumerate}
\end{eg}

\subsection{Regulator} 
Let $A$ be a semi-local commutative ring and let $Q(A)$ be the total ring of fractions of $A$. 

\begin{defn}
\label{defn:regulator}
We call a perfect complex $E$ of $A$-modules \emph{rationally acyclic} if $E \otimes_A Q(A)$ is acyclic. In that case, by Example \ref{eg:acyclic complex}, there is a canonical isomorphism
\[
\reg_E: {\det}_{Q(A)} (E \otimes_A Q(A)) \isoto Q(A)
\]
that we call the \emph{regulator} of $E$. Precomposing with the canonical map $\det_A(E) \to \det_{Q(A)}(E \otimes_A Q(A))$, we obtain a map $\reg_E:\det_A(E) \to Q(A)$ that we also call the regulator.

For any generator $x \in \det_A(E)$, the fractional ideal $\reg_E(x)A$ is independent of the choice of $x$, and we call it $\reg(E)$, the \emph{regulator ideal} of $E$.
\end{defn}

\begin{eg}
Suppose $E = [A e_0 \xrightarrow{\lambda} A e_{1}]$ (so $e_i$ is in degree $i$) and that $\lambda \in A$ is a non-zero divisor. Then $E$ is rationally acyclic, and $\frac{e_0}{\lambda e_1}$ is the basis of $\det_{Q(A)}(E \otimes_A Q(A))$ that induces $\reg_E$, so $\reg_E(e_0/e_1) = \lambda$.
\end{eg} 

\begin{eg}
\label{eg:regular of a square}
Let $E=[A e_0  \xrightarrow{\sv{a}{b}} A e_1 \oplus A f_1 \xrightarrow{(c,d)} A e_2]$, and suppose that $E$ is rationally acyclic. In particular, there is a vector $\sv{c'}{d'} \in Q(A)^2$ such that $cc'+dd'=1$, and the set of such vectors is a torsor under translation by $Q(A)\cdot \sv{a}{b}$.

Then the basis of $\det_{Q(A)}(E \otimes_A Q(A))$ that induces $\reg_E$ is 
\[
\frac{e_0e_2}{(ad'-c'b)e_1\wedge f_1}
\]
(note that this is independent of the choice of $\sv{c'}{d'}$)
so $\reg_E(\frac{e_0e_2}{e_1\wedge f_1})=(ad'-c'b)$.
\end{eg}

Note that, if we have a short exact sequence of complexes
\[
0 \to E_1 \to E_2 \to E_3 \to 0
\]
with each $E_i$ perfect and rationally acyclic, then the composition
\[
{\det}_{Q(A)}(E_2 \otimes_A Q(A)) \isoto {\det}_{Q(A)}(E_1 \otimes_A Q(A)) \otimes_{Q(A)} {\det}_{Q(A)}(E_3\otimes_A Q(A)) \xrightarrow{\reg} Q(A)
\] 
coincides with the regulator of $E_2$.

\subsection{Regulator and lifting}
Let $(A,\m_A,k)$ be a noetherian local ring and let $E$ be a perfect complex of $A$-modules. Assuming that $E$ has a certain special form, we show that $\reg(E)$ can be used to determine the lifting behavior of classes in $H^1(E \otimes_A k)$. We also give a simple cohomological criterion for when a complex has this special form.

\begin{lem}
\label{lem:lifting}
Suppose that $E= [A \xrightarrow{\lambda} A][-1]$ with $\lambda \in \m_A \cap Q(A)^\times$. Then, for any proper ideal $\mathfrak{a} \subset A$, the map
\[
H^1(E \otimes_A A/\mathfrak{a}) \to H^1(E \otimes_A k)
\]
is surjective if and only if $\lambda \in \mathfrak{a}$.
\end{lem}
\begin{proof}
We have $H^1(E \otimes_A A/\mathfrak{a}) = \ker( A/\mathfrak{a} \xrightarrow{\lambda} A/\mathfrak{a})$ and $H^1(E \otimes_A k)=k$, and the map is induced by the quotient $A/\mathfrak{a} \to A/\m_A =k$. 

If $\lambda \in \mathfrak{a}$, then $H^1(E \otimes_A A/\mathfrak{a})=A/\mathfrak{a}$ and the map is clearly surjective. On the other hand, if the map is surjective, then there is $x \in A$ with $x \not\in \mathfrak{m}_A$ such that $x\lambda \in \mathfrak{a}$. Since $A$ is local, this implies that $x \in A^\times$, so $\lambda \in \mathfrak{a}$.
\end{proof}

We give a criterion for the conditions of the lemma to be satisfied.

\begin{lem}
\label{lem:simple complex}
Let $E$ be a perfect, rationally acyclic complex of $A$-modules. Assume that
\begin{enumerate}
\item $\dim_{k} H^i(E \otimes_A^\mathbb{L} k)=1$ for $i=1,2$,
\item $H^i(E) =0$ for $i \ne 1,2$.
\end{enumerate}
Then there is a quasi-isomorphism 
 $E \simeq [A \xrightarrow{\lambda} A][-1]$ with $\lambda \in \mathfrak{m}_A \cap Q(A)^\times$.
\end{lem}
\begin{proof}
Since $E$ is perfect, we can assume, without loss of generality, that $E$ is a bounded complex of finitely generated projective $A$-modules. 
Since $H^i(E)=0$ for $i>2$, we can further assume that $E^i=0$ for $i>2$. Then the map
\[
H^2(E)\otimes_A k \to H^2(E \otimes_A^\mathbb{L} k) 
\]
is an isomorphism. By (1) and Nakayama's lemma, $H^2(E)$ is cyclic as a $A$-module. Choose a surjection $A \to H^2(E)$ and lift it to a map $A \to Z^2(E)$. This defines a map of complexes
\[
A[-2] \to E.
\]
Let $C=\mathrm{Cone}(A[-2] \to E)$. By construction, $H^i(C)=0$ for $i>1$, so, just as we argued above for $E$, we see that
\[
H^1(C)\otimes_A k \to H^1(C \otimes_A^\mathbb{L} k) 
\]
is an isomorphism. Considering the triangle obtained by applying $(-) \otimes_A^\mathbb{L} k$ to
\[
A[-2] \to E \to C
\]
we can see that
\[
H^1(E \otimes_A^\mathbb{L} k) \to H^1(C \otimes_A^\mathbb{L} k)
\]
is an isomorphism, so $\dim_{k}H^1(C \otimes_A^\mathbb{L} k)=1$. By Nakayama's lemma, $H^1(C)$ is cyclic as an $A$-module. Choose a surjection $A \to H^1(C)$ and lift it to a map $A \to Z^1(C) = \ker(A \oplus E^1 \to E^2)$, and let $A \xrightarrow{\lambda} A$ be the composition of $A \to Z^1(C)$ with the natural map $Z^1(C) \to A$.
This defines a map of complexes
\[
[A \xrightarrow{\lambda} A][-1] \to E
\]
that induces a surjection on $H^1$ and an isomorphism on all other $H^i$. Hence we have $Q(A)/\lambda Q(A) \cong H^2(E \otimes_A^\mathbb{L} Q(A))$, which is zero since $E$ is rationally acyclic. This implies that $\lambda \in Q(A)^\times$, so $\lambda$ is a non-zero divisor and $H^1([A \xrightarrow{\lambda} A][-1])$ is $0$. Since the map is surjection on $H^1$, this implies that $H^1(E)$ is zero as well, and hence that the map $[A \xrightarrow{\lambda} A][-1] \to E$ is a quasi-isomorphism.
\end{proof}

\section{Galois cohomology}
\label{app:galois}

\subsection{Notation for Galois cohomology}
In this section, we fix notation for various Galois cohomology complexes. We follow the notation used by Flach in \cite{flach2004}.

 The continuous group cohomology $H^*(G,-)$ of a topological group is computed by the complex $C(G,-)$ of continuous cochains. We let $\RG(G,-)$ denote the class of $C(G,-)$ in the derived category.

Let $S$ be a finite set of primes and let $G_{\Q,S}$ to be the Galois group of the maximal extension of $\Q$ that is unramified outside $S$. We denote $\RG(G_{\Q,S},-)$ by $\RG(\Z[1/S],-)$ (this makes sense because $G_{\Q,S}$ is the \'etale fundamental group of $\Z[1/S]$). Similarly, we let $\RG(\Q_\ell,-)$ denote $\RG(G_{\Q_\ell},-)$ for $\ell \in S$.

For $\ell \in S$, we let
\[
\RG_{(\ell)}(\Z[1/S],- ) := \mathrm{Cone}\left(\RG(\Z[1/S],-) \to \RG(\Q_\ell,-)\right)[-1]
\]
denote the Selmer complex with trivial-at-$\ell$ condition, and let
\[
\RG_{c}(\Z[1/S],- ) := \mathrm{Cone}\left(\RG(\Z[1/S],-) \to \bigoplus_{\ell \in S} \RG(\Q_\ell,-)\right)[-1]
\]
denote the `compactly-supported' cohomology complex.

Now we define the local finite-cohomology complex $\RG_f(\Q_s,-)$ for $s \in S$. First suppose $s=\ell$, a finite prime, and $\ell \ne p$. If $M$ is a pro-$p$ abelian group with a continuous action of $G_{\Q_\ell}$, we will denote by $\RG_f(\Q_\ell,M)$ the complex
\[
M^{I_\ell} \xrightarrow{1-\Fr_\ell} M^{I_\ell}.
\]
For $s=\infty$, we define $\RG_f(\R,M) = \RG(\R,M)$.
If instead $M$ is a finite-dimensional $\Q_p$-vector space with a continuous action of $G_{\Q_p}$, we will denote by $\RG_f(\Q_p,M)$ the complex
\[
D_\mathrm{crys}(M) \xrightarrow{(1-\Fr_p,\mathrm{id})} D_\mathrm{crys}(M)\oplus D_\mathrm{dR}(M)/ D^0_\mathrm{dR}(M).
\]

For $s \in S$, we define the local ``non-finite"  cohomology complex
\[
\RG_{/ f}(\Q_s,M) = \mathrm{Cone}(\RG_f(\Q_s,M) \to \RG(\Q_s,M))
\]
for $M$ as in the previous paragraph.

Assuming that $S$ contains $p$, $\infty$ and any prime where $M$ ramifies, we define the Bloch-Kato Selmer complex $\RG_f(\Z[1/S],M)$ to be
\[
\RG_f(\Z[1/S],M) = \mathrm{Cone}\left(\RG(\Z[1/S],M) \to \oplus_{s \in S} \RG_{/ f}(\Q_s,M)\right)[-1].
\]
We have the triangle
\[
\RG_f(\Z[1/S],M) \to \RG_c(\Z[1/S],M) \to \oplus_{s \in S} \RG_{f}(\Q_s,M).
\]

By convention, when we write $\RG(T \otimes \Lambda)$ for some module $T$, the $G_{\Q,Np}$-action on $\Lambda$ is via the universal character. It is known that, if $T$ is perfect as a $\Lambda$-module, then $\RG(T \otimes \Lambda)$ and $\RG(\Q_\ell,T \otimes \Lambda)$ are perfect complexes of $\Lambda$-modules (see \cite[Proposition 4.17]{kato1993a} or \cite[Proposition 1.6.5]{FK2006}).

\subsection{A complex that computes tame local Galois cohomology}
The following lemma is surely well-known in some form. It is essentially how one computes cohomology of a semi-direct product of cyclic groups.
\begin{lem}
\label{lem:local complex} Let $N$ and $p$ be distinct primes. Let $R$ be a $\Z_p$-algebra, and let $M$ be an $R$-module that is finitely generated as $\Z_p$-module with continuous tamely-ramified action of $G_{\Q_N}$ (i.e.~the inertia group acts through its pro-$p$ quotient). Let $\gamma \in G_{\Q_N^\mathrm{nr}}$ be an element that topologically generates the maximal pro-$p$ quotient, and let $\mathrm{Fr}_N \in G_{\Q_N}$ be a Frobenius element. Then there is an isomorphism in the derived category of $R$-modules
\[
\RG(\Q_N,M) \cong 
\left[ \vcenter{\xymatrix{
M \ar[r]^-{1-\gamma} \ar[d]_-{1-\mathrm{Fr}_N^{-1}} & M \ar[d]^-{1-\mathrm{Fr}_N^{-1}\mathcal{N}} \\
M \ar[r]^-{1-\gamma}  & M
}}\right]
\]
where $\mathcal{N}=\sum_{i=0}^{N-1} \gamma^i$.
\end{lem}
\begin{proof}
We have an isomorphism
\[
\RG(\Q_N,M)\cong \RG(\F_N,\RG(\Q_N^\mathrm{nr},M)).
\]
where we are identifying $\Gal(\Q_N^\mathrm{nr}/\Q_N) =G_{\F_N}$ via $\mathrm{Fr}_N$.

We first compute $\RG(\Q_N^\mathrm{nr},M)$. Let $G_{\Q_N^\mathrm{nr}} \onto G_{\Q_N^\mathrm{nr}}^{\mathrm{pro}-p}$ be the maximal pro-$p$ quotient, and let $G_{\Q_N^\mathrm{nr}}^{\mathrm{non}-p}$ the the kernel of this quotient. Then we have
\[
\RG(\Q_N^\mathrm{nr},M) \cong \RG(G_{\Q_N^\mathrm{nr}}^{\mathrm{pro}-p},\RG(G_{\Q_N^\mathrm{nr}}^{\mathrm{non}-p},M)).
\]
Since $M$ is pro-$p$, we have $\RG(G_{\Q_N^\mathrm{nr}}^{\mathrm{non}-p},M) = R^0\Gamma(G_{\Q_N^\mathrm{nr}}^{\mathrm{non}-p},M)$, which is simply $M$, since $M$ is tamely-ramified. Hence we have
\[
\RG(\Q_N^\mathrm{nr},M) \cong \RG(G_{\Q_N^\mathrm{nr}}^{\mathrm{pro}-p},M).
\]
Since $G_{\Q_N^\mathrm{nr}}^{\mathrm{pro}-p}$ is topologically generated by the image of $\gamma$, we have
\[
 \RG(G_{\Q_N^\mathrm{nr}}^{\mathrm{pro}-p},M) \cong [B^1(G_{\Q_N^\mathrm{nr}}^{\mathrm{pro}-p},M) \to Z^1(G_{\Q_N^\mathrm{nr}}^{\mathrm{pro}-p},M)]
\]
Let $M^\flat$ be the $R[G_{\F_N}]$-module that is $M$ is an $R$-module, but with $\mathrm{Fr}_N^{-1}$ acting by 
\[
\mathrm{Fr}_N^{-1} \cdot m^\flat := \left(\mathrm{Fr}_N^{-1} \sum_{i=0}^{N-1} \gamma^i\right) m.
\]
We have an isomorphism $Z^1(G_{\Q_N^\mathrm{nr}}^{\mathrm{pro}-p},M) \isoto M^\flat$  of $R[G_{\F_N}]$-modules by $f \mapsto f(\gamma)$. Indeed, $f$ is determined by $f(\gamma)$ as we have
\[
f(\gamma^n) = \sum_{i=0}^{n-1} \gamma^i f(\gamma)
\]
as can be proven by induction using the cocycle property. Moreover, we have
\begin{align*}
(\mathrm{Fr}_N^{-1} \cdot f)(\gamma) & = \mathrm{Fr}_N^{-1} f(\mathrm{Fr}_N \gamma \mathrm{Fr}_N^{-1}) \\
& =  \mathrm{Fr}_N^{-1} f(\gamma^N) \\
& = \left(\mathrm{Fr}_N^{-1} \sum_{i=0}^{N-1} \gamma^i\right) f(\gamma).
\end{align*}
Similarly, we have an isomorphism $M \isoto B^1(G_{\Q_N^\mathrm{nr}}^{\mathrm{pro}-p},M)$ of $R[G_{\F_N}]$-modules, given by $m \mapsto (g \mapsto (g-1)m)$. Under these isomorphisms, the inclusion $B^1(G_{\Q_N^\mathrm{nr}}^{\mathrm{pro}-p},M) \to Z^1(G_{\Q_N^\mathrm{nr}}^{\mathrm{pro}-p},M)$ is identified with $M \xrightarrow{1-\gamma} M^\flat$.

Hence we have
\[
 \RG(G_{\Q_N^\mathrm{nr}}^{\mathrm{pro}-p},M) \cong [M \xrightarrow{1-\gamma} M^\flat]
\]
in the derived category of $R[G_{\F_N}]$-modules.

Now, for complex of $R[G_{\F_N}]$-modules $M'$, we have
\[
\RG(\F_N,M') \cong [M' \xrightarrow{1-\mathrm{Fr}_N^{-1}} M']
\]
in the derived category of $R$-modules. Hence we have
\begin{align*}
\RG(\Q_N,M) & \cong \RG(\F_N,\RG(\Q_N^\mathrm{nr},M)) \\
& \cong \RG(\F_N,\RG(G_{\Q_N^\mathrm{nr}}^{\mathrm{pro}-p},M)) \\
& \cong \RG(\F_N, [M \xrightarrow{1-\gamma} M^\flat]) \\
& \cong \left[ \vcenter{\xymatrix{
M \ar[r]^-{1-\gamma} \ar[d]_-{1-\mathrm{Fr}_N^{-1}} & M^\flat  \ar[d]^-{1-\mathrm{Fr}_N^{-1}} \\
M \ar[r]^-{1-\gamma}  & M^\flat 
}}\right] \\
& \cong \left[ \vcenter{\xymatrix{
M \ar[r]^-{1-\gamma} \ar[d]_-{1-\mathrm{Fr}_N^{-1}} & M \ar[d]^-{1-\mathrm{Fr}_N^{-1}\mathcal{N}} \\
M \ar[r]^-{1-\gamma}  & M
}}\right].
\end{align*}
\end{proof}
\bibliographystyle{alpha}
\bibliography{oct2018}

\end{document}